\documentclass[11.05pt,a4paper]{amsart}
\usepackage{rotating}
\usepackage[all]{xy}

\usepackage[utf8]{inputenc}
\usepackage{amscd,amssymb,amsopn,amsmath,amsthm,graphics,amsfonts,enumerate,verbatim,calc}
\headsep=0.5cm \numberwithin{equation}{section}
\newtheorem{theorem}{Theorem}[section]
\newtheorem{lemma}[theorem]{Lemma}\newtheorem{observation}[theorem]{Observation}
\newtheorem{proposition}[theorem]{Proposition}
\newtheorem{discussion}[theorem]{Discussion}

\newtheorem{corollary}[theorem]{Corollary}\newtheorem{notation}[theorem]{Notation}

\theoremstyle{definition}
\newtheorem{definition}[theorem]{Definition}\newtheorem{fact}[theorem]{Fact}
\theoremstyle{remark}
\newtheorem{remark}[theorem]{Remark}\newtheorem{problem}[theorem]{Problem}
\newtheorem{example}[theorem]{Example}
\newtheorem{question}[theorem]{Question}
\newtheorem{conjecture}[theorem]{Conjecture}
\newtheorem{acknowledgement}{Acknowledgement}

\newcommand{\HH}{\operatorname{H}}
\newcommand{\DVR}{\operatorname{DVR}}

\newcommand{\Ass}{\operatorname{Ass}}

\newcommand{\im}{\operatorname{im}}

\newcommand{\grade}{\operatorname{grade}}
\newcommand{\pic}{\operatorname{Pic}}
\newcommand{\Cl}{\operatorname{Cl}}

\newcommand{\UFD}{\operatorname{UFD}}

\newcommand{\Spec}{\operatorname{Spec}}

\newcommand{\type}{\operatorname{type}}

\newcommand{\Ht}{\operatorname{ht}}
\newcommand{\pd}{\operatorname{p.dim}}
\newcommand{\gd}{\operatorname{gl.dim}}

\newcommand{\Proj}{\operatorname{Proj}}
\newcommand{\Syz}{\operatorname{Syz}}

\newcommand{\BNSI}{\operatorname{BNSI}}

\newcommand{\rank}{\operatorname{rank}}
\newcommand{\Gdim}{\operatorname{Gdim}}

\newcommand{\V}{\operatorname{V}}

\newcommand{\id}{\operatorname{id}}
\newcommand{\Ext}{\operatorname{Ext}}
\newcommand{\Se}{\operatorname{S}}

\newcommand{\Sym}{\operatorname{Sym}}
\newcommand{\R}{\operatorname{R}}
\newcommand{\Supp}{\operatorname{Supp}}

\newcommand{\Tor}{\operatorname{Tor}}

\newcommand{\Hom}{\operatorname{Hom}}

\newcommand{\Ann}{\operatorname{Ann}}

\newcommand{\codepth}{\operatorname{codepth}}
\newcommand{\depth}{\operatorname{depth}}

\newcommand{\emd}{\operatorname{emb}}
\newcommand{\Soc}{\operatorname{Soc}}
\newcommand{\Char}{\operatorname{char}}
\newcommand{\coker}{\operatorname{coker}}

\newcommand{\D}{ D}

\newcommand{\vil}{\operatornamewithlimits{\varinjlim}}

\newcommand{\lo}{\longrightarrow}
\newcommand{\fm}{\frak{m}}
\newcommand{\fp}{\frak{p}}
\newcommand{\fq}{\frak{q}}
\newcommand{\fa}{\frak{a}}

\newcommand{\fn}{\frak{n}}

\newcommand{\bm}{\begin{matrix}}
\newcommand{\dm}{\end{matrix}}

\newcommand{\PP}{\mathbb{P}}

\begin{document}

\author[]{mohsen asgharzadeh}

\address{}
\email{mohsenasgharzadeh@gmail.com}

\title[ ]
{reflexivity revisited}

\subjclass[2010]{ Primary 13D02; 13C10}
\keywords{Betti numbers;   free modules; reflexive modules; splitting. }
\dedicatory{}

\begin{abstract}We study some aspects of reflexive modules. For example,
we search conditions for which reflexive modules are free or   close to  free modules.
\end{abstract}

\maketitle

\setcounter{tocdepth}{1}
\tableofcontents

\section{Introduction}

In this paper $(R,\fm,k)$ is a commutative Noetherian local ring and $M$ is a finitely generated $R$-module, unless otherwise specified. The notation $\mathcal{M}$ stands for a general module. For simplicity, the notation $\mathcal{M}^\ast$ stands for $\Hom_R(\mathcal{M},R)$. Then $\mathcal{M}$ is called \textit{reflexive} if the natural map $\varphi_{\mathcal{M}}:\mathcal{M}\to\mathcal{M}^{\ast\ast}$ is bijective. Finitely generated projective modules are reflexive. In his seminal paper \cite{kap}, Kaplansky proved that projective modules over local rings are free. The local assumption is really important. Indeed, there are many interesting research papers, and even books, on the freeness of projective modules over polynomial rings with coefficients from a field. In general, the class of reflexive modules is extremely large compared to the projective modules.

As a generalization of Seshadri's result, Serre observed in 1958 that over $2$-dimensional regular local rings, finitely generated reflexive modules are free. This result has some applications; for instance, in the arithmetic properties of Iwasawa algebras (see \cite{ser}).

It seems the freeness of reflexive modules is subtle even over very special rings. For example, in \cite[Page 518]{Lam}, Lam says that the only obvious examples of reflexive modules over $R:=\frac{k[X,Y]}{(X,Y)^2}$ are the free modules $R^n$. Ramras posed the following:

\begin{problem}[See \cite{ramg}, Page 380]
	When are finitely generated reflexive modules free?
\end{problem}

Over quasi-reduced rings, Problem 1.1 was completely answered (see Proposition \ref{4.1}). Ramras proved any finitely generated reflexive module over $\BNSI$ (Betti numbers strictly increasing) rings is free. We present some applications of this result. Also, we introduce the class of \textit{eventually $\BNSI$} rings and study freeness of reflexive modules over them. We know any nonzero free module decomposes into a direct sum of rank-one submodules. Treger conjectured (see \cite[Page 462]{t}):

\begin{conjecture}
	Let $(R,\fm,k)$ be a complete local (singular and containing a field) normal domain of dimension $2$ where $k=\overline{k}$ and $\Char(k) \neq 2$. Then $R$ is a cone such as $\frac{k[[x,y,z]]}{(x^2+y^2+z^2)}$ if and only if every nonzero reflexive module $M$ decomposes into a direct sum of rank-one submodules.
\end{conjecture}

Section 2 connects reflexivity to finiteness conditions including finite generation. As an application, we study the following:

\begin{question}
	When are quasi-reflexive modules flat?
\end{question}

Section 3 collects different notions of reflexivity. It may be worth mentioning that there are connections to the miracle of set theory with application to homological reflexivity. For instance, see the seminal work of Shelah \cite{sh}.

Section 4 deals with the freeness of reflexive modules over some classes of rings. This section is divided into five subsections: the first subsection deals with the freeness of reflexive modules. The second is about the freeness of totally reflexive modules. In Subsection 4.3, we investigate the freeness of weakly Gorenstein modules. The fourth subsection deals with the freeness of $M^\ast$ compared to the freeness of $M$. We present modern proofs of a result by Bass \cite[5.2]{bass2}. Our arguments are independent of Bass's work and appear to be original. Subsection 4.5 is devoted to the freeness of specialized (resp. generalized) reflexive modules.

In Section 5 we investigate the reflexivity of (multi)-dual modules, i.e., we study some homological aspects. For instance, see Corollary \ref{ca} and Propositions \ref{gf2} and \ref{gfu}. The new invariant is $\ell(M^{n\ast})$ and its asymptotic behavior, namely we study the existence of $\lim_{n\to\infty} \frac{\ell(M^{n\ast})}{\type(R)^{n}}$ and detect some numerical invariants when the limit exists. As a sample we present the following:

\begin{observation}
	Let $(R,\fm)$ be such that $\fm^2=0$ and $M$ be a finitely generated module with no free direct summands. Then 
	\[
	\lim_{n\to\infty} \frac{\ell(M^{n\ast})}{\type(R)^{n}}=\beta_0(M).
	\]
\end{observation}

In Section 6 we settle Conjecture 1.2. In Section 7 we deal with a question of Braun:

\begin{question}
	Let $I\lhd R$ be a reflexive ideal of a normal domain with $\operatorname{id}_R(I) < \infty$. Is $I\simeq\omega_R$?
\end{question}

It may be worth mentioning that the results of Section 7 have an essential role in the complete solution of Braun's question; see the recent preprint \cite{ufd}.

In Section 8 we descend freeness (resp. reflexivity) from the endomorphism ring to the module. This is inspired by the paper of Auslander-Goldman. Similarly, we descend some data from the higher tensor products to the module. In particular, we slightly extend some results of Vasconcelos, Huneke-Wiegand, and the recent work of \v{C}esnavi\v{c}ius. For example, see Corollary \ref{decproc1}.

Grothendieck solved a conjecture of Samuel, see Theorem \ref{groth}. In Section 9 we try to understand these miracles by looking at the mentioned result of Auslander-Goldman. We do this by using some results of Section 8. In particular, there is a connection between Problem 1.1 and the $\UFD$ property of regular (complete-intersection) rings, see Corollary \ref{regufd} as a sample. Samuel remarked that there is no symmetric analogue of Auslander's theorem on the torsion part of tensor power modules over regular rings. In this regard, we present a tiny remark:

\begin{observation}
	Let $(R,\fm)$ be a regular local ring and $M$ be of rank one. If $\Sym_n(M)$ is reflexive for some $n\geq \max\{2,\dim R\}$, then $M$ is free.
\end{observation}

Also, see Corollary \ref{ob3}.

In Section 10 we investigate the reflexivity of ideals. The motivation is a result of Bass \cite[Theorem 6.2]{bass} which says that an Artinian local ring is Gorenstein iff all of its ideals are reflexive. We remark that reflexivity of the maximal ideal does not imply reflexivity of all ideals; see Discussion \ref{n}. In particular, we connect to a recent result of Faber \cite{F}. Then, we extend the mentioned theorem of Bass by showing that an Artinian local ring is Gorenstein iff its maximal ideal is reflexive.

In the final section, we give a positive answer to some questions asked by Holanda and Miranda-Neto \cite{nano} by using freeness of some reflexive modules via some local cohomological arguments, which may initially seem unrelated. We apply this theme to present more applications to \cite{nano}. For more details, see Observations \ref{11.4}--\ref{11.8}.

We provide numerous examples and remarks, along with an extensive list of references. Despite this effort, it is possible that some excellent references have been overlooked. This also highlights the challenges of writing a comprehensive paper on reflexive modules.

\section{Reflexivity and  finiteness}

All  rings are noetherian.
Following Bass \cite{bass2}, $\mathcal{M}$ is called \textit{torsion-less} if  $\varphi_{\mathcal{M}}$ is injective
(this some times is called semi-reflexive).
A torsion-less module $\mathcal{M}$ is noetherian if and only if $\mathcal{M}^\ast$ is noetherian.
Submodules of a torsion-less module are torsion-less.
We say $\mathcal{M}$  is \textit{weakly reflexive}  if  $\varphi_{\mathcal{M}}$ is surjective.
In general, neither submodule nor quotient of a weakly reflexive module is weakly reflexive.

\begin{observation}\label{obz}
Let $(R,\fm)$ be  a zero-dimensional Gorenstein local ring.  Then the properties of weakly reflexive and finitely generated are the same.
In particular, (weakly) reflexivity is not closed under taking direct limits.
\end{observation}

\begin{proof}Here we use the concept of Gorenstein-projective. For its definition, see \cite[4.2.1]{kris}.
Since  $R$  is zero-dimensional and Gorenstein, any module is Gorenstein-projective (see \cite[4.4.8]{kris}).
It follows by definition that any module is a submodule of a projective module.
Over local rings, and by the celebrated theorem of Kaplansky, any projective module is free.
Combining these, any module is a submodule  of a free module.
It follows by definition that any module is torsion-less.
Now, let $\mathcal{M}$ be weakly reflexive.
Thus, $\mathcal{M}$ is reflexive.
It is shown in \cite[Corollary 2.4.(4)]{sch}
that over any commutative artinian ring, every reflexive module is finitely
generated. By this, $\mathcal{M}$ is finitely
generated. Conversely, assume that $M$ is finitely
generated. Since  $R$  is zero-dimensional and Gorenstein, $M$ is
reflexive.
To see the particular case, we remark that any module can be written as a directed limit of finitely generated modules.
We use this along with the first part to get the claim.
\end{proof}

\begin{discussion}
The zero dimensional assumption is important. Indeed, by \cite[Ex. 2.8'(2)]{Lam}, $\bigoplus_\mathbb{N} \mathbb{Z}$ is reflexive.
This is one of two extra-credit exercises in the book \cite{Lam}.
\end{discussion}

 There are flat modules that are not reflexive, e.g. the vector space $\bigoplus_\mathbb{N} \mathbb{Q}$ and the abelian group $\mathbb{Q}$. In order to handle this drawback, let
 $R$ be a  normal domain of dimension bigger than zero with fraction field $Q(R)$. Following Samuel,  $\mathcal{M}$ is called
\textit{quasi-reflexive} if $\mathcal{M}=\bigcap_{\fp\in \Spec(R),\Ht(\fp)=1}\mathcal{M}_{\fp}$ where we compute the intersection in $\mathcal{M}_0:=\mathcal{M}\otimes_RQ(R)$.
Yuan proved that any flat module  is quasi-reflexive (see \cite[Lemma 2]{yuan}).

\begin{lemma}
Let $(R,\fm)$ be  a  normal domain  of dimension bigger than zero and   $\mathcal{M}$ be quasi-reflexive.  Then
there is a family of finitely generated and reflexive modules $\{M_i\}_{i\in I}$ such that
$\mathcal{M}={\varinjlim}M_i$.
\end{lemma}

\begin{proof}
 There is a directed family $\{N_i\}_{i\in I}$ of finitely generated submodules of $\mathcal{M}$ such that $\mathcal{M}={\varinjlim}N_i$. Let $$\Spec^1(R):=\{\fp\in \Spec(R),\Ht(\fp)=1\}.$$
 For each $i$, we set $M_i:=\bigcap_{\fp\in \Spec^1(R)}(N_i)_{\fp}$ where we compute the intersection in $\mathcal{M}_0$. Let $\fp\in\Spec^1(R)$.
 We note that $M_i\subset(N_i)_{\fp}\subset\mathcal{M}_{\fp}$ and
 so $M_i\subset \bigcap_{\fp\in \Spec^1(R)}\mathcal{M}_{\fp}=\mathcal{M}$. Suppose $N_i\subset N_j $. It follows that $(N_i)_{\fp}\subset (N_j)_{\fp} \subset\mathcal{M}_{\fp}$.
 Consequently, $$\bigcap_{\fp\in \Spec^1(R)}(N_i)_{\fp}\subset\bigcap_{\fp\in \Spec^1(R)} (N_j)_{\fp} \subset\bigcap_{\fp\in \Spec^1(R)}\mathcal{M}_{\fp}=\mathcal{M}.$$
This says that
 $M_i\subset M_j \subset\mathcal{M}$. Let $x\in\mathcal{M}$. There is
 a $j\in I$ such that $x\in N_j$. Since $N_j\subseteq M_j$, we have $x\in M_j\subset{\varinjlim}M_j$, i.e., $\mathcal{M}\subseteq{\varinjlim}M_i$. The reverse inclusion
 holds by definition. In sum,
 $\mathcal{M}={\varinjlim}M_i$. Clearly,
  $N_i^{\ast\ast}$ is finitely generated. In view of Fact \ref{5.1v} (see below) $N_i^{\ast\ast}$ is reflexive.
 Again, let $\fp\in \Spec^1(R)$. Note that $R_{\fp}$ is a discrete valuation ring. Over  such a
 ring any torsion-free is free, and so reflexive. Since $N_i$ is finitely generated, $\Hom_R(N_i,R)$ commutes with the localization. Thus, $$(N_i)_{\fp}\simeq(N_i)_{\fp}^{\ast\ast}\simeq(N_i^{\ast\ast})_{\fp}.$$
Thanks to \cite[Proposition 1]{yuan},  $$N_i^{\ast\ast}\simeq\bigcap_{\fp\in \Spec^1(R)}(N_i^{\ast\ast})_{\fp}.$$ We put these together,
 $$N_i^{\ast\ast}\simeq\bigcap_{\fp\in \Spec^1(R)}(N_i^{\ast\ast})_{\fp}\simeq\bigcap_{\fp\in \Spec^1(R)}(N_i)_{\fp}.$$
 From this, $M_i\simeq N_i^{\ast\ast}$. In particular,
   $M_i$ is finitely generated and reflexive.
 This completes the proof.
\end{proof}

\begin{proposition}\label{pp}
Let $(R,\fm)$ be  a  regular local ring of dimension at most two.
There is no difference between flat modules and quasi-reflexive modules. In particular,
any direct limit of quasi-reflexive modules is quasi-reflexive.
\end{proposition}

\begin{proof}
Recall that any flat module  is quasi-reflexive.  Conversely, let $\mathcal{M}$ be quasi-reflexive. By the above lemma,
there is a family of finitely generated reflexive modules $\{M_i\}$ such that
$\mathcal{M}={\varinjlim}M_i$. By the mentioned  result of Serre, each $M_i$ is free.
Clearly, direct limit of free modules is flat. We apply this to observe that $\mathcal{M}$ is flat.
To see the particular case, we mention that  direct limit of flat modules is again flat.
\end{proof}

\begin{corollary}\label{qreff}
Let $(R,\fm,k)$ be  a normal domain  of dimension bigger than zero. The following are equivalent:
\begin{enumerate}
\item[i)] there is no difference between flat modules and quasi-reflexive modules,
\item[ii)]  $R$ is  a  regular local ring of dimension at most two.
\end{enumerate}
\end{corollary}

\begin{proof}$i)\Rightarrow ii)$: The second syzygy of $k$ is reflexive and so quasi-reflexive. By the assumption it is flat. Finitely generated flat modules
over local rings are free. Thus,
second syzygy of $k$ is free. This in turn is equivalent with $\pd(k)\leq 2$.
In view of local-global-principle, $R$ is    regular and of dimension at most two.

$ii)\Rightarrow i)$: This is in   Proposition \ref{pp}.
\end{proof}

In the next section we recall the concept of weakly Gorenstein for finitely generated  modules over local rings
and in \S 4 we study their freeness. It seems the origin of this comes back to 1950 when  Whitehead
 posed a problem: Let $G$ be an abelian group such that $\Ext^{+}(G,\mathbb{Z})=0$. Is $G$ free?
Shelah \cite{sh}  proved that this is undecidable in ZFC. By ZFC we mean the Zermelo--Fraenkel set theory with the axiom of choice included. There are a lot of notions between
 free abelian groups and reflexive abelian groups if we have no finiteness restrictions:
 
 \[\begin{array}{lllllll}
 &&\textmd{free}\Rightarrow\textmd{hereditarily separable}\Rightarrow
 \textmd{coseparable}\Rightarrow\textmd{separable}\Rightarrow\aleph_1\textmd{-free} \Rightarrow\textmd{torsionless},
 \end{array}\]
 and that $$\aleph_1 \textmd{-coseparable} \Rightarrow\textmd{reflexive}\Rightarrow\textmd{torsionless}.$$ To see these, we cite the book \cite{ek}, and record a question: \begin{question}Let $G,H$ be reflexive abelian groups. Is $G\otimes_\mathbb{Z} H$ reflexive?
 \end{question}   

 Let us connect to \cite{s} and cite the following  analytic theme of reflexivity:
 \begin{fact} Let $V$ be a Banach space. Then $\Hom(V, \mathbb{R})$ is a Smith space. If $W$ is
a Smith space, then $\Hom(W,  \mathbb{R})$ is a Banach space, where in both cases we endow the dual space
with the compact-open topology,  with the corresponding bi duality maps are isomorphisms.\end{fact}

\begin{conjecture}(Shelah \cite[Conjecture 1.6]{shref})
There are reflexive groups of any cardinality.
\end{conjecture}
\section{Homological reflexivity}

Here, modules are finitely generated.
We  recall different notions of reflexivity (our reference book on this topic is \cite{kris}).
A reflexive module $M$ is called \textit{totally reflexive} if $\Ext^i(M,R)=\Ext^i(M^\ast,R)=0$ for all $i>0$.
By $\mu(M)$ we mean the minimal number of generators of $M$.

\begin{observation}\label{m} Let $(R,\fm)$ be  a  local ring such that $\fm^2=0$.  Then any finitely generated weakly reflexive module is  totally  reflexive.\end{observation}

\begin{proof}
Let $M$ be weakly reflexive. Suppose first that $\mu(\fm)=1$. It follows that $R$ is zero-dimensional Gorenstein ring. Over such a ring any finitely generated module
is totally reflexive. Suppose now that $\mu(\fm)>1$. Let $\D(-)$ be the Auslander transpose. The  cokernel of $(-)\to(-)^{\ast\ast}$ is $\Ext^2_R(\D(-),R)$.
Since $M$ is weakly reflexive, $\Ext^2_R(\D(M),R)=0$. Menzin proved that $\Ext^i_R(\D(M),R)=0$ for some $i>1$ is equivalent
to the freeness (see \S 4 for more details).
\end{proof}
Recall that a module $M$ is called \textit{weakly Gorenstein}
if $\Ext^i_R(M,R)=0$  for all $i>0$, see e.g. \cite{mas}.

\begin{definition}
	A  local ring $(R,\fm)$ is called  \it{quasi-reduced} if it satisfies  Serre's condition $(\Se_1)$ and be generically Gorenstein.
\end{definition}
 
\begin{remark}
	Let $(R,\fm)$ be a quasi-reduced local ring. The following are equivalent:
	\begin{enumerate}
		\item[i)] any reflexive module is totally reflexive,
		\item[ii)] any reflexive module is  weakly Gorenstein,
		\item[iii)]  $R$ is a Gorenstein ring of dimension at most two.
	\end{enumerate}
\end{remark}
\begin{proof}$i)\Rightarrow ii)$:  This is clear.

	$ii)\Rightarrow iii)$: Let $M$ be reflexive.  By a result of Hartshorne (see Fact \ref{mas} below) $M^\ast$ is reflexive. By our assumption, $M$ and $M^\ast$ are weakly Gorenstein. By definition,
	$\Ext^+_R(M,R)=\Ext^+_R(M^\ast,R)=0$. Thus,
	 $M$ is totally reflexive. 
	 Over quasi-reduced rings, second syzygy modules are reflexive.
	From  this, $\Syz_2(k)$ is reflexive. We proved that $\Syz_2(k)$ is totally reflexive.
	By this, $R$ is a Gorenstein ring
	of dimension at most two.

	$iii)\Rightarrow i)$: This is clear by Auslander-Bridger formula.
\end{proof}

The quasi-reduced assumption is important, see Observation 4.1.
Also, we say
 $M$ is  \textit{skew Gorenstein} if $\Ext^i_R(M^\ast,R)=0$ for all $i>0$.

\begin{definition}
We say $M$ is \textit{homologically reflexive} if $\Ext^i(M,R)=\Ext^i(M^\ast,R)=0$ for all $i>0$.
Also, $M$ is called   strongly reflexive if $\Ext^i_R(\D(M),R)=0$   for all $i>0$ (in the common terminology: $M$ is n-torsionless for all $n$).
\end{definition}

\begin{example}\label{2.5}
Let  $(R,\fm,k)$ be zero-dimensional but not Gorenstein. Then $k$ is torsion-less but neither reflexive nor weakly Gorenstein.
 \end{example}

 \begin{proof}
 Since   $R$ is not Gorenstein, $\dim_k(\Soc(R))>1$. Also, $0\neq\Soc(R)$  is equipped with a  vector space structure. The same thing holds for its dual.
In particular, the natural embedding   $k\hookrightarrow\Soc(R)^\ast=k^{\ast\ast}$ is not surjective. By definition, $k$ is torsion-less and  $k$ is not reflexive. Suppose on the contrary that $k$ is weakly Gorenstein.  The condition $\Ext^i_R(k,R)=0$ for all $i>0$  implies that $R$ is Gorenstein. This  is excluded from the assumption. So, $k$ is not weakly Gorenstein.
\end{proof}

\begin{definition}
An $R$-module $M$ is called  infinite syzygy if there is an exact sequence
$$	0 \lo M \lo F_0 \lo \ldots \lo F_n \lo F_{n+1} \lo \ldots$$
	where $F_i$
	is free.
\end{definition}

We note that  any module is torsion-less if and only if any module is  infinite syzygy.
Over artinian rings we can talk a little more:

\begin{corollary}\label{INF}Let $(R,\fm,k)$ be local and zero-dimensional. The following are equivalent:
		\begin{enumerate}
		\item[i)]  any finitely generated module is torsion-less,
		\item[ii)] any finitely generated module is  infinite syzygy,
		\item[iii)]  the canonical module $\omega_R$ is torsion-less,
		\item[iv)]  $R$ is a Gorenstein ring.
	\end{enumerate}
\end{corollary}
\begin{proof} 
	 $i)\Rightarrow ii) \Rightarrow iii)$: These are trivial.
	 
 $iii) \Rightarrow i)$: Recall that $\omega_R=E(k)$. Note that any module
	 can be embedded into  injective modules.
	 Any injective module is a direct sum of $E(k)$. By assumption $\omega_R$ is torsion-less.
	 Thus, $\omega_R$ is submodule of a free module.
	 It follows that any  module can be embedded into   free modules.
	 Therefore, any finitely generated module is torsion-less.

	  $ii) \Leftrightarrow iv)$: See \cite[Proposition 2.9]{miller}. Here, we present a shorter argument. Suppose $\omega_R$ is torsionless. So, it is  a submodule of a free module $F$. Indeed, $0\to\omega_R\stackrel{\subseteq}\lo F\to\coker({\subseteq})\to 0$ splits, as $\omega$ is injective. So,  $\omega$ is free. Consequently, $R$ is Gorenstein.
\end{proof}
\begin{question}(Iyengar)
	Suppose  $k$ is an infinite syzygy. Is $R$ zero-dimensional and Gorenstein?
\end{question}
\section{Being free}
This section is divided into 5 subsections: 
 The first deals with freeness of reflexive modules.  
The second subsection  is about of freeness of  totally reflexive modules.
The third subsection  investigates  freeness  of weakly Gorenstein modules.
 4.4 deals with freeness  of $M^\ast$ against the freeness of $M$.
The final subsection  is devoted to the   freeness  of specialized (resp. generalized)  reflexive  modules.

\subsection{Freeness of reflexive modules}We present three different arguments for Lam's prediction.
One of them implicitly is in the PhD thesis of Gover \cite{gover}. 
\begin{observation}\label{lamnew}
	Let $k$ be an algebraically closed field and
	let  $R:=\frac{k[X_,Y]}{(X,Y)^2}$.
	Then the only  examples of nonzero reflexive modules
	are the free modules $R^\ell$ for some $\ell\in \mathbb{N}$.
\end{observation}

\begin{proof}   Let $\mathcal{M}$ be reflexive. Recall that over any commutative artinian ring, every reflexive module is finitely
	generated. Hence, $\mathcal{M}$ is finitely generated. Without loss of the generality we may assume that $\mathcal{M}$ is indecomposable.
	Several years ago, these modules were classified by  Kronecker. We use a presentation given by Harthshorne    (see  \cite[Page 60]{har2}).  We  find  that  any  indecomposable  finitely  generated   is isomorphic  to  one of the following: 
	\begin{enumerate}
		\item[i)]  $M_n:=\fm^{n-1}/ \fm^{n+1}$,  
		\item[ii)] $N_{n,p}:=M_n/ (p^{n})$ where $p\in\PP^1_k$ representing a linear  form $p=ax+by$  with $a,b\in k$, or
		\item[iii)] $W_n=M_n/ (x^n,y^n )$.
	\end{enumerate}
	Note that $M_1=R$ (resp. $M_{>2}=R$) which is free. Also,  $M_2=\fm$. We claim that $\fm$ is not reflexive. Indeed, $\fm$ is a vector space of dimension $2$.
	It follows that  $\fm^{\ast}$ is a vector space of dimension $4$ and so $\fm^{\ast\ast}$ is a vector space of dimension $8$.  From this,
	$\fm$ is not reflexive. Now, we compute $N_{1,p}^{\ast\ast}$. To this end,
	$$N_{1,p}^{\ast}=\{r\in R:rp=0\}=\fm.$$ Hence $N_{1,p}^{\ast}\cong k\oplus k$. 
Since $p\in\fm$ and $\fm^2=0$, we see $N_{1,p}=R/(p)$ is a vector space and  $\dim_k(N_{1,p})=2$. From this, $$N_{1,p}^{\ast\ast}=\fm^\ast=k^{4\oplus}\supsetneqq N_{1,p}.$$This implies that  
	$N_{1,p}$  is not reflexive. The same proof shows that $N_{2,p}$  is not reflexive.
	Recall that $W_1=k$. Also,  $W_1^{\ast\ast}=k^{4\oplus}\supsetneqq k =W_{1}$ and so 
	$W_{1}$  is not reflexive. The same proof shows that $W_{2}$  is not reflexive.
\end{proof}

The following is the second proof of Lam's prediction and   is implicitly in Gover's thesis\footnote{His beautiful argument appeared in the proof of the following fact: Let $(S,\fn)$ be a regular local ring of dimension 2. Then $R:=S/ \fn^2$ is not Gorenstein. One may prove this by only saying that type of $R$ is two.}:

\begin{discussion}
 Let $k$ be any field and
	let  $R:=\frac{k[X_,Y]}{(X,Y)^2}$.
	Then the only  examples of nonzero reflexive modules
	are the free modules $R^\ell$ for some $\ell\in \mathbb{N}$.
\end{discussion}

\begin{proof} 
	Let $M$ be reflexive. Since the ring is artinian,   we may assume that $M$ is finitely generated.  It is easy to show that any finitely generated module such as  $M$ can be written
	$	M = M_t\oplus F$ where $F$ is free and $M_t^\ast=\Hom(M_t,\fm) $. Since $\fm^2=0$, $\fm$ is a vector space of dimension $2$. We apply this observation to see
	$$M_t^\ast=\Hom(M_t,\fm)=\oplus_2 \Hom(M_t,k).$$ Take another duality,
	$$M_t^{\ast\ast}=\Hom(M_t,\fm)^\ast=\oplus_2 \Hom(M_t,k)^\ast=\oplus_2 \Hom(\Hom(M_t,k),\fm)=:\oplus_2 \Hom(V,\fm)$$since $V:=\Hom(M_t,k)$ does not contain a free direct summand. Hence  $M_t^{\ast\ast}= \oplus_4 \Hom(V,k)$. Following Gover, we 
	take  another duality,
	and  this shows that   $$M_t^{\ast\ast\ast}= \oplus_8 \Hom_k(\Hom_k(V,k),k).$$
	Direct summand and dual of reflexive modules is again reflexive. So,  $M_t^{\ast\ast\ast}\cong M_t^\ast$. From this we get that $$2\dim V=8\dim V.$$ In other words $V=0$,
	and consequently $M_t^\ast=0$. Take another duality, $M_t\cong M_t^{\ast\ast}=0$. In particular, $M$ is free.
\end{proof}
Recall that the $i^{th}$ betti number of $M$ is given by $\beta_i(M):=\dim_k\Tor^R_i(k,M)$.
Suppose $x\in \fm$ is nonzero. Note that $R\stackrel{x}\lo R\to R/xR\to 0$ is a part of minimal free resolution.
By definition, $\beta_0(R/xR)=\beta_1(R/xR)$. By this reason, a ring is called $\BNSI$  if  for every non-free module $M$ we have $\beta_{i}(M)>{\beta_{i-1}}(M)$ for  $i> 1$.
We will use the following result several times:

\begin{fact}\label{www}(See \cite[2.4 and 2.5]{ram})
Let $R$ be $\BNSI$. Suppose $\Ext^i_R(M,R)=0$ for some $i\geq 2$. Then $M$ is free.
Also, finitely generated weakly reflexive modules are free.
\end{fact}

The following  extends Lam's prediction:

\begin{corollary}\label{lam}
Let $(R,\fm)$ be  a  local ring such that $\mu(\fm)>1$ and suppose $R$ is one of the following types:
 \begin{enumerate}
\item[i)]  $R:=\frac{k[X_1,\ldots,X_m]}{(X_1,\ldots,X_m)^n}$, or
\item[ii)] $R$ is such that $\fm^2=0$.
\end{enumerate}
Then the only  examples of nonzero reflexive modules
are the free modules $R^\ell$ for some $\ell\in \mathbb{N}$. Also, over $\frac{k[X]}{(X)^n}$ there is no difference between finitely generated modules and reflexive modules.
\end{corollary}

  \begin{proof}   Let $\mathcal{M}$ be reflexive. Recall that over any commutative artinian ring, every reflexive module is finitely
generated. Hence, $\mathcal{M}$ is finitely generated. By  \cite[3.3]{ram}, $\frac{k[X_1,\ldots,X_m]}{(X_1,\ldots,X_m)^n}$ is $\BNSI$ provided that $m>1$. Also, in the second case,
 the ring $R$ is $\BNSI$, see the proof of Observation \ref{m}.  In the light of Fact \ref{www}, $\mathcal{M}$ is free and of finite rank.
For the last claim its enough to note that $\frac{k[X]}{(X)^n}$ is zero-dimensional and Gorenstein (see Observation \ref{obz}).
\end{proof}

The following extends \cite[Proposition 2 and 4]{men}  and \cite[Proposition 2.4]{y} by a new proof.

\begin{corollary}\label{m2COR}
Let $(R,\fm)$ be  a  local ring such that $\fm^2=0$. If $\mu(\fm)\neq1$, then  $R$ is $\BNSI$. In particular,
\begin{enumerate}
\item[i)]  if $\Ext^i_R(M,R)=0$ for some $i\geq 2$, then $M$ is free,
\item[ii)] any finitely generated  weakly reflexive module  is free,
\item[iii)] there is no difference between  reflexive modules  and free modules of finite rank,
\item[iv)] there is no difference between  skew Gorenstein modules  and free modules of finite rank,
\item[v)] there is no difference between  strongly reflexive modules  and free modules of finite rank.
\end{enumerate}\end{corollary}

\begin{proof}
 Since $\fm^2=0$ and
$\fm\neq 0$, we have $\fm \subseteq(0:\fm)\subseteq \fm$, i.e., $(0:\fm)=\fm$. We apply this to see  $$\dim_k\left(\frac{(0:\fm)+\fm^2}{\fm^2}\right)=\dim_k(\frac{\fm}{\fm^2})=\mu(\fm)>1.$$  Let $M$ be nonfree. By Auslander-Buchsbaum formula, $\pd(M)=\infty$ because $\depth(R)=0$. In particular, all betti numbers are nonzero. We recall from  \cite[Proposition 2.4]{fun} that
 $$\beta_{i+1}(M)\geq \left(\dim\frac{(0:\fm)+\fm^2}{\fm^2}\right)\beta_{i}(M)\quad \forall i\geq1.$$ Since $\beta_{i}(M)\neq 0$ and $\dim_k\frac{(0:\fm)+\fm^2}{\fm^2}>1$  we have  $$\left(\dim_k\frac{(0:\fm)+\fm^2}{\fm^2}\right)\beta_{i}(M)>\beta_{i}(M).$$
 We combine these to see $$\beta_{i+1}(M)\geq \left(\dim_k\frac{(0:\fm)+\fm^2}{\fm^2}\right)\beta_{i}(M)>\beta_{i}(M)$$ for all $i\geq1$.
By definition,
$R$ is $\BNSI$. Fact \ref{www} yields i) and ii). Let $\mathcal{M}$ be reflexive. Recall that over any commutative artinian ring, every reflexive module is finitely
generated. Hence, $\mathcal{M}$ is finitely generated. By  Fact \ref{www}, we get iii). Let $M$ be skew Gorenstein. By i), $M^\ast$ is  free.
In general freeness of $(-)$ can not follow from $(-)^\ast$. But, if depth of a ring is zero this   happens (see \cite[Lemma 2.6]{ram}).
Since $\depth(R)=0$, we conclude  that $M$ is free. It remains to prove v): This is trivial, because strongly reflexive
modules are reflexive and by iii) reflexive modules are free.
\end{proof}

Let $(R,\fm,k)$ be any Gorenstein ring with $\fm^3=0$ which is not field.  Then $k$ is totally reflexive but it is not free.
To find nonfree totally reflexive modules over non-Gorenstein ring  with $\fm^3=0$ it is enough to look at \cite{y}.
The following  extends and corrects  \cite[Examples (3)]{h} where it is shown that weakly Gorenstein modules are free without assuming $\fm^2\neq(0:\fm)$.

\begin{corollary}\label{les}
Let $(R,\fm)$ be a local ring such that $\fm^3=0$, $\mu(\fm)>1$  and that $\fm^2\neq(0:\fm)$. Then $R$ is $\BNSI$.
For finitely generated modules we have:\[\begin{array}{lllllll}
&&\textmd{weakly Gorenstein}&\Leftrightarrow\textmd{homologically reflexive}\Leftrightarrow\textmd{totally reflexive}\\
&&\quad\quad\quad\quad
\begin{turn}{90}
$\Leftrightarrow$
\end{turn}\\
&&\textmd{strongly reflexive} &\Leftrightarrow\textmd{reflexive}\Leftrightarrow
\textmd{weakly reflxive}\Leftrightarrow\textmd{skew Gorenstein} \Leftrightarrow\textmd{free}
\end{array}\]
\end{corollary}

\begin{proof}
The claim in the case $\fm^2=0$ is in Corollary \ref{m2COR} (here, we used the assumption $\mu(\fm)>1$).
Without loss of the generality we may assume that $\fm^2\neq 0$.
Let $M$ be non-free. In view of  \cite[Lemma 3.9]{les}  $\beta_{i}(M)>{\beta_{i-1}}(M)$ for  $i> 0$. By definition,
$R$ is $\BNSI$. In view of   Fact  \ref{www}  we see all of the equivalences except the skew Gorenstein property.
Suppose $M$ is skew Gorenstein. In the same vein and similar to  Corollary \ref{m2COR}, we see that  $M$ is free.
\end{proof}

Freeness of weakly Gorenstein modules over $\frac{\mathbb{Q}[X,Y]}{(X^2,XY,Y^3)}$ follows from  \cite[Corollary 4.8]{saeed2}.
What can one say on the freeness of  reflexive modules? We have:

\begin{example}\label{2113}
Let $R:=\frac{\mathbb{Q}[X,Y]}{(X^2,XY,Y^3)}$. Then any reflexive module is free.
\end{example}

\begin{proof} Let $\mathcal{M}$ be any reflexive. Since $\dim R=0$, $\mathcal{M}$ is finitely generated.
We have $\fm=(x,y)$ and that  $\fm^3=(x^3,x^2y,xy^2,y^3)=0$. Let us compute
the socle. By definition, $$(0:\fm)=(0:x)\cap(0:y)=(x,y)\cap(x,y^2)=(x,y^2).$$
Clearly, $(0:\fm)\neq \fm^2$. In view of Corollary \ref{les}, $\mathcal{M}$ is free.
 \end{proof}

Freeness of totally reflexive is not enough strong to deduce freeness of reflexive:

\begin{example}\label{211}
Let $R:=\frac{\mathbb{Q}[[X,Y,Z]]}{(XY,XZ)}$. Then totally reflexive are free and there is a reflexive module which is not free.
\end{example}

\begin{proof}The ring  $R$ is  the fiber product $\mathbb{Q}[[X]]\times_{\mathbb{Q}}\mathbb{Q}[[Y,Z]]$ (see \cite[Ex. 3.9]{saeed1}).
Since the ring is not Gorenstein (e.g. it is not equi-dimensional) and is fiber product, any totally reflexive is free, see \cite[Corollary 4.8]{saeed2}. Now, we look at $M:=\frac{R}{(y,z)}$.
We have, $$\Hom(M,R)\simeq\{r\in R:r(y,z)=0\}\simeq xR\simeq R/(0:x)= R/(y,z)=M.$$From this, $M\simeq M^{\ast\ast}$. By \cite[1.1.9(b)]{kris} a finitely generated module is reflexive if it is isomorphic to its bidual. So, $M$ is reflexive. Since it annihilated by $(y,z)$ we see that it is not free.
 \end{proof}

\subsection{Freeness of totally reflexive modules}

\begin{definition}
We say a local ring is eventually $\BNSI$ if there is an $\ell\geq1$ such that for every non-free module $M$ we have $\beta_{i}(M)>{\beta_{i-1}}(M)$ for  $i> \ell$.
\end{definition}

\begin{example}\label{2.1}
Let $(A,\fn)$ be a regular local ring of dimension $n>1$ and $f\in\fn$ be nonzero.
Let $R:=\frac{A}{f\fn}$. The following assertions hold:
\begin{enumerate}
\item[i)] Any finitely generated  totally reflexive module is free.
\item[ii)] For every non-free module $M$ we have: $$\ldots>\beta_{n+2}(M)>{\beta_{n+1}}(M)>\beta_{n}(M)\geq\beta_{n-1}(M)\geq\ldots\geq\beta_{1}(M).$$
\item[iii)] Let $\D(-)$ be the Auslander transpose and let $M:=\D(fR)$. Then \[\begin{array}{ll}
\ldots>\beta_{n+2}(M)>{\beta_{n+1}}(M)&>\beta_{n}(M)\\
&\geq\beta_{n-1}(M)\geq\ldots\geq\beta_{3}(M)=n\\
&>\beta_{2}(M)=\beta_{1}(M)=1\\
&< n=\beta_{0}(M) .\end{array}\]\footnote{The paper \cite[3.5]{ab} claims that $\beta_{i+1}(-)>\beta_{i}(-)$ for all $i>\depth(R)$ over any  Golod ring $R$ which is not hypersurface and all modules of infinite projective dimension. This is in contradiction with iii).  I feel that
\cite{ab} has a misprint and  the mentioned result should be   stated that $\beta_{i+1}(-)>\beta_{i}(-)$ for all $i>\emd(R)$. }
In fact, $fR$ is non-free but reflexive.
\end{enumerate}In particular,  the ring $R$ is eventually $\BNSI$ but is not  $\BNSI$.
\end{example}

The fact that $fR$ is  reflexive is due to Ramras by a different argument. Also, part i) extends \cite[Example 5.1]{ta} in three directions via a new argument.

\begin{proof}  i)  The ring $R$ is Golod. Note that $\codepth(R):=\emd(R)-\depth R=n-0>1$.
Rings of codepth at most one are called hypersurface.
Avramov and Martsinkovsky proved,
over a Golod local
ring that is not hypersurface, that every  any module of finite $\Gdim$  is of finite $\pd$.
 Also, this is well-known that $\Gdim$  is the same as of  $\pd$ provided $\pd$ is finite.
 By this, any finitely generated  totally reflexive module is free.

ii)  Since $\depth(R)=0$ any non-free is of infinite projective dimension. Lescot proved over a Golod ring $R$ which is not hypersurface,  $\beta_{i}(M)>{\beta_{i-1}}(M)$ for  $i> \emd(R)$ and all  $M$ of infinite projective dimension (see \cite[6.5]{les2}). Also \cite[Proposition 3.4]{ram} says that the betti sequence is not decreasing. By these,
$$\beta_{n+i}(M)>\ldots>{\beta_{n+1}}(M)>\beta_{n}(M)\geq\beta_{n-1}(M)\geq\ldots\geq\beta_{1}(M).$$

iii)  By \cite[Proposition 3.5]{ram}, $R$ is not $\BNSI$.  It follows from ii) that there is  a non-free module $N$ and some $i<n$ such that $\beta_{i+1}(N)=\beta_{i}(N)$.
Let us find them. Let  $\fn:=(X_1,\ldots,X_n)$. Since $(0:_Rf)=\fm$ the minimal presentation of $fR$ is  $R^n\to R\to fR\to 0$. Apply $\Hom(-,R)$ we have
$$0\lo (fR)^{\ast}\lo R^{\ast}\simeq R\lo R^n\simeq (R^n)^{\ast}\lo \D(fR)\lo 0\quad(\star)$$
Then $R\lo R^n$ is the minimal presentation of  $\D(fR)$ if $(fR)^{\ast}$ has no free direct summand.
Since $\depth R=0$ and in the light of \cite[Lemma 2.6]{ram}, $(fR)^{\ast}$  has no  free direct summand.
This implies that $R\stackrel{\varphi}\lo  R^n \to \D(fR)\to 0 $ is the minimal presentation of $ \D(fR)$, where $\varphi:=(x_1,\ldots,x_n)$. Since $$\ker(\varphi)=\bigcap \Ann(x_i)=(0:\fm)=fR \quad(\star,\star)$$ we see that
 $R\stackrel{f}\lo R\stackrel{\varphi}\lo  R^n \to \D(fR)\to 0$  is part of the minimal free resolution  of $ \D(fR)$.
That is $$\beta_{2}(\D(fR))=\beta_{1}(\D(fR))=1< n=\beta_{0}(\D(fR)).$$ Similarly, $$\beta_{3}(\D(fR))=\mu(0:f)=\mu(\fm)=n.$$
It follows from $(\star)$ and $(\star,\star)$ that $(fR)^\ast= fR$. Consequently, $fR$ is  reflexive. Since $\fm$ annihilated $fR$ we deduce that $fR$ is not free.

 By ii) the ring is  eventually $\BNSI$. By iii) $R$ is   not  $\BNSI$.
\end{proof}

\begin{lemma}\label{ab1}
Let $(R,\fm)$ be a local   ring of depth zero. Suppose there is $\ell$ such that the betti sequence $\beta_{i}(M)>{\beta_{i-1}}(M)$ for  $i> \ell$ and all nonfree totally reflexive module $M$.
Then  $\beta_{i}(M)>{\beta_{i-1}}(M)$ for  $i> 0$.
\end{lemma}

\begin{proof}The idea is taken from \cite{ram2}.
We look at the minimal free resolution $$\cdots  \lo  R^{\beta_{i}{({M^\ast})}}\lo\cdots\lo  R^{\beta_{0}{({M^\ast})}}\lo M ^{\ast}\lo 0 .$$ Since $M$ is totally reflexive
$\Ext^{+}(M^{\ast},R)=0$. By duality, there is an exact sequence  $$0\lo M^{\ast\ast}\simeq M\lo R^{\beta_{0}{({M^\ast})}}\stackrel{d_0}\lo R^{\beta_{1}{({M^\ast})}}\stackrel{d_1}\lo\ldots$$ We set $N:=\ker(d_{\ell+1})$. We note $\beta_{i}(M)=\beta_{i+\ell}(N)$. Since $M$ is totally reflexive and in view of $$0\lo  M\lo R^{\beta_{0}}\lo \ldots \lo R^{\beta_{\ell}}\lo N\to 0,$$ we deduce that  $N$ is of finite G-dimension. It follows from Auslander-Bridger-formula that  $N$ is totally  reflexive. By the assumption, $\beta_{i}(N)>{\beta_{i-1}}(N)$ for  $i> \ell$. From this we get $\beta_{i}(M)>{\beta_{i-1}}(M)$ for  $i> 0$.
\end{proof}

\begin{remark}\label{dd}
Depth of any eventually $\BNSI$ ring is zero. Indeed, suppose there is a regular element $x$. In view of the exact sequence $0\to R\to R\to R/xR\to 0$ we see that $\beta_{i}(M)={\beta_{i-1}}(M)=0$ for  $i> 1$, a contradiction.
\end{remark}

Problem 1.1 specializes (see \cite[Page  402]{ab}):
When totally reflexive modules are free?

\begin{corollary}\label{ab2}
Let $(R,\fm)$ be  eventually $\BNSI$.
Then any  finitely generated totally reflexive module is  free.
\end{corollary}

Example \ref{2.1} shows that one can not replace totally-reflexivity with the reflexivity.
This can follows from \cite{ab}. However, the following proof is so easy:

\begin{proof}There is $\ell$ such that  $\beta_{i}(M)>{\beta_{i-1}}(M)$ for  $i> \ell$ and all  non-free $M$. Due to the  above  remark, $\depth R=0$. In view of Lemma \ref{ab1} we can take $\ell=0$ provided $M$ is totally reflexive and non-free.
 Since $\depth R=0$ and in the light of \cite[Lemma 2.6]{ram}, $M^{\ast}$ is not free.  Since $M^{\ast}$ is not free, $\D(M)$ is not free. Let $F_1\to F_0\to M\to 0$ be a minimal presentation of $M$.
We observed that $$0\lo M^{\ast}\lo F_0^{\ast}\lo F_1^{\ast}\lo \D(M)\lo 0 $$  provides a minimal presentation of $\D(M)$. Recall that  $\Gdim(-) = 0 $ if and only if $\Gdim(\D(-)) = 0$.
By this
 $\D(M)$ is totally reflexive. Also, $$\beta_{0}(M^\ast)=\beta_{2}(\D(M))>\beta_{1}(\D(M))=\beta_{0}(M).$$
 Another  use of the former observation implies that $$\beta_{0}(M)=\beta_{0}(M^{\ast\ast})>\beta_{0}(M^\ast)>\beta_{0}(M).$$ This is a contradiction that we searched for it.
\end{proof}

The above argument shows a little more:
\begin{remark}\label{ab3}
Let $(R,\fm)$ be a local   ring. Suppose there is $\ell$ such that   $\beta_{i}(M)>{\beta_{i-1}}(M)$ for  $i> \ell$ and all nonfree totally reflexive module $M$.
 Then, any  finitely generated totally reflexive module is  free.
\end{remark}

\begin{proof}We assume $\depth(R)>0$. Let $M$ be totally reflexive. Suppose on the contradiction that
$M$ is not free. It turns out that $M$ is of infinite projective dimension.
Without loss of generality we may assume that $M$ has no free direct summand
(any  direct summand  of weakly reflexive module is weakly reflexive).
 If $M^\ast$ has a free direct summand, then its dual $M^{\ast\ast}$ has a free direct summand, and in view of  $M\simeq M^{\ast\ast}$ we get a contradiction. We proved that $M^\ast$ has no free direct summand. Without loss of generality we may assume that $\ell> \depth R$.
We look at the minimal free resolution $$\cdots  \lo  R^{\beta_{i}{({M^\ast})}}\lo\cdots\lo  R^{\beta_{0}{({M^\ast})}}\lo M ^{\ast}\lo 0 .$$ Since $M$ is totally reflexive
$\Ext^{+}(M^{\ast},R)=0$. By duality, there is an exact sequence  $$0\to M^{\ast\ast}\simeq M\lo R^{\beta_{0}{({M^\ast})}}\stackrel{d_0}\lo R^{\beta_{1}{({M^\ast})}}\stackrel{d_1}\lo\cdots.$$  We set $N:=\ker(d_{\ell+1})$. Since $M$ is of finite $\Gdim$ and in view of $$0\lo  M\lo R^{\beta_{0}}\lo \ldots \lo R^{\beta_{\ell}}\lo N\lo 0,$$ we deduce that  $N$ is of finite G-dimension.  Due to the exact sequence
we have $\Ext^i_R(N,R)=\Ext^{\ell-i}(M^\ast,R)$ for all $0<i\leq\depth R$, this is zero  because $M$ is  totally reflexive. Recall that $$\Gdim(N)=\sup\{i:\Ext^i_R(N,R)\neq 0\}$$ and  that $\Gdim$ is bounded by depth provided it is finite. We combine these to see that $\Gdim(N)=0$.
Also, $\pd(N)=\infty$, because
 $\pd(M)=\infty$.  By the assumption, $\beta_{i}(N)>{\beta_{i-1}}(N)$ for  $i> \ell$. From this we get $\beta_{i}(M)>{\beta_{i-1}}(M)$ for  $i> 0$. In particular, we can take $\ell=0$.
Suppose on the contradiction that  $\pd(M^\ast)<\infty$. It follows that  $$\pd(M^\ast)=\sup\{i:\Ext^i_R(M^\ast,R)\neq 0\}.$$ This is zero because $M$ is totally reflexive. Freeness
pass to dual. From this,
$M\simeq M^{\ast\ast}$ is free,  a contradiction.
Let $F_1\to F_0\to M\to 0$ be a minimal free resolution of $M$.
Then we have $0\to M^{\ast}\to F_0^{\ast}\to F_1^{\ast}\to \D(M)\to 0$. Apply this to see $\pd(\D(M))=\infty$. Also, $F_0^{\ast}\to F_1^{\ast}\to \D(M)\to 0 $ is the minimal presentation of $\D(M)$, because   $M^\ast$ has no free direct summand.    Recall that  $\Gdim(\D(M)) = 0 $.
We have  $$\beta_{0}(M^\ast)=\beta_{2}(\D(M))>\beta_{1}(\D(M))=\beta_{0}(M).$$
Another  use of the former observation implies that $\beta_{0}(M)=\beta_{0}(M^{\ast\ast})>\beta_{0}(M^\ast)>\beta_{0}(M)$, a contradiction.
\end{proof}

One may like a ring for which   every non-free module $M$ there is an $\ell(M)$ such that  $\beta_{i}(M)>{\beta_{i-1}}(M)$ for  $i> \ell(M)$.
We say a such ring is \textit{weakly $\BNSI$}.
This property is not enough strong to deduce freeness from the totally reflexiveness:

\begin{observation}\label{m3}
Let $(R,\fm)$ be  a  Gorenstein  local ring such that $\fm^3=0$ and  $\mu(\fm)>2$.
Then $(R,\fm)$ is weakly $\BNSI$. Also, there is a nonfree totally reflexive module.
\end{observation}

\begin{proof}
If $\fm^2$  were be zero then we should have $(0:\fm)=\fm$. Since the ring is Gorenstein,  it follows that $\mu(\fm)=1$. This excluded by the assumption. We   assume that $\fm^2\neq0$.
 We set $n:=\mu(\fm)$. Since $\fm^3=0$ we see that $0\neq\fm^2\subset(0:\fm)$.
 Since $\dim(\Soc(R))=1$, we have
 $\ell(\fm^2)=1$.   It follows that $$\ell(R)=1+\ell(\fm)=1+\ell(\fm/ \fm^2)+\ell(\fm^2)=n+2.$$
Recall from  \cite[Proposition 2.2]{ga} that
\begin{enumerate}
\item[Fact] A) Let $(A,\fn)$ be an artinian ring and  $N$ be   finitely generated.  Let $h$ be the smallest $i$ such that $\fn^{i+1}=0$. Then
$\beta_{i+1} (N)\geq(2 \mu(\fn)- \ell(A)+h-1)\beta_{i} (N)$  for all $n \geq\mu(N)$.
\end{enumerate} Since $\fm^3=0$ we  have $h=2$. Also, $$
2 \mu(\fm)- \ell(R)+h-1= 2n-(n+2)+2-1 =n-1 \geq2.$$ Let $M$ be nonfree. By Auslander-Buchsbaum formula,
$\beta_{i} (M)\neq 0$ for all $i$.
In view of Fact A) we see $$\beta_{i+1} (M)\geq2\beta_{i} (M)>\beta_{i}(M) \emph{ for all } i \geq\mu(M).$$   By definition,
$R$ is weakly $\BNSI$.
Every non-free module $M$ is totally reflexive (e.g. the residue field), because the ring is Gorenstein.
\end{proof}

\begin{example}
The ring $R:=\frac{\mathbb{Q}[[X,Y,Z]]}{(X^2-Y^2,Y^2-Z^2,XY,YZ,ZX)}$ is weakly $\BNSI$.
\end{example}

\begin{proof}
This is a folklore example of  a Gorenstein ring. Also, $\fm^3=0$ and  $\mu(\fm)=3$. By Observation \ref{m3} $R$ is weakly $\BNSI$.
\end{proof}

\subsection{Freeness of weakly Gorenstein modules}

 Here,   modules are finitely generated.

\begin{observation}\label{m3COR}(After Menzin-Yoshino)
Let $(R,\fm,k)$ be non-Gorenstein Cohen-Macaulay ring   of minimal multiplicity and $k$ be infinite. The following assertions hold:

i) If $\Ext^i_R(M,R)=0$ for all $1\leq i\leq 2\dim R+2$, then $M$ is free.
In particular,  \[\begin{array}{lllllll} \textmd{strongly
reflexive}
\Leftrightarrow\textmd{homologically reflexive}
\Leftrightarrow\textmd{weakly Gorenstein}
\Leftrightarrow\textmd{totally reflexive}
\Leftrightarrow\textmd{free}
\end{array}\]

ii) The modules in part i) are equivalent with skew Gorenstein if and only if $\dim R=0$.

iii) Suppose in addition that the ring is  complete and quasinormal. Then
there is a nonfree reflexive module.
\end{observation}

It may be nice to give an example of (iii): For example, any 2-dimensional non-Gorenstein normal local domain
with a rational singularity (see \cite[Example 4.8]{saeed1}).

\begin{proof} i) The first claim is in \cite[Proposition 7]{men}.
In particular, $$
\textmd{homologically reflexive}\Leftrightarrow\textmd{weakly Gorenstein}\Leftrightarrow\textmd{totally reflexive}\Leftrightarrow\textmd{free}.
$$
Suppose  $M$ is strongly
reflexive. Since $\Ext^{+}(M^\ast,R)=0$ and in view of the first part we see $M^\ast$ is free. This yields the freeness of  $M^{\ast\ast}$.
By the assumption, $M\simeq M^{\ast\ast}$. Consequently, $M$ is free.

ii) If $R$ is artinian, then $\fm^2=0$ and desired claim is in Corollary \ref{m2COR}.
Suppose $R$ is not artinian. Due to the Cohen-Macaulay assumption, $\depth (R)=\dim R>0$. We look at $M:=R/ \fm\oplus R$. It follows that $M^\ast\simeq R$. Hence, $\Ext^{+}_R(M^\ast,R)=0$. Thus,  $M$ is skew Gorenstein. Clearly, $M$ is not free.

iii) Since $R   $ is Cohen-Macaulay and homomorphic image of a Gorenstein ring
it posses a canonical module $\omega_R$.   In general, canonical module is not reflexive. However,  there
is a situation for which canonical module is reflexive. Indeed,  Vasconcelos proved that:
\begin{enumerate}
\item[Fact A)]  Over quasi-normal rings, a necessary and sufficient condition for $M$ to be reflexive
is that every $R$-sequence of two or less elements be also an $M$-sequence.
\end{enumerate}
The canonical module is maximal Cohen-Macaulay. Due to Fact A) we see $\omega_R$ is reflexive. However, $\omega_R$ is not free, because the ring is not    Gorenstein.
\end{proof}

Following \cite{saeed1}, we say $\fm$  is \textit{quasi-decomposable} if $\fm$ contains an $R$-sequence $\underline{x}$
(the empty set allowed) such that the module $\fm /\underline{x}R$  decomposes into nonzero submodules.
The following may extend \cite[Corollary 6.8]{saeed1} (and \cite[4.7]{saeed2}) by presenting a bound:

\begin{proposition}\label{nt}
Let $(R,\fm,k)$ be a non-Gorenstein local ring such that $\fm$ is quasi-decomposable. The following  holds

\begin{enumerate}
\item[i)] If $\Ext^i_R(M,R)=0$ for all $1\leq i\leq 6+3\depth R$, then $M$ is free.
In particular,   \[\begin{array}{lllllll} \textmd{strongly
reflexive}
\Leftrightarrow\textmd{homologically reflexive}
\Leftrightarrow\textmd{weakly Gorenstein}
\Leftrightarrow\textmd{totally reflexive}
\Leftrightarrow\textmd{free}
\end{array}\]
\item[ii)] There are situations for which reflexive are free.
\item[iii)] Suppose in addition
that $R$ is complete and quasi-normal, there are  reflexive modules
that are not  free.
\end{enumerate}
\end{proposition}

\begin{proof} i)
Suppose first that $\fm=I\oplus J$. This translates to   $R\simeq R/ I\times_k R/J$. In view of \cite[Corollary 6.3]{saeed1} we see
$\pd(M)\leq 1$. If  projective dimension a module $(-)$ were be finite, then it should be  equal to $\sup\{i:\Ext^i_R(-,R)\neq 0\}$. We call this property by $(\ast)$. From this,
$M$ is free. Now, suppose that there is a nonempty set  $\underline{x}=x_1,\cdots, x_n$ of
$R$-sequence such that the module $\fm /\underline{x}R$ is decomposable. Let $n:=\mu(M)$ and look at $0\to\Syz(M)\to R^n\to M\to 0$.
Note that $x_1$ is regular over $\Syz(M)$. Let $\overline{(-)}:=-\otimes R/x_1R$. By the standard reduction,  $$\Ext^i_{\overline{R}}(\overline{\Syz}(M),\overline{R})=0  \quad \forall  1\leq i\leq 6+3\depth R-2,$$ see \cite[Proposition 7]{men} for more details. Following the inductive hypothesis, $\overline{\Syz}(M)$ is projective over $\overline{R}$.
Let $m:=\mu(\Syz(M))$ and apply $\overline{(-)}$ to $0\to\Syz_2(M)\to R^m\to \Syz(M)\to 0$. This is in turn imply that $\overline{\Syz}_2(M)=0$. We apply  Nakayama's lemma to conclude that
$\Syz_2(M)=0$. By definition,
 $\Syz(M)$ is projective, and consequently $\pd_{R}(M)\leq1$. Again, we use $(\ast)$ to deduce that $\pd_{R}(M)=0$.

ii) Remark that $R_1:=k[x,y]/(x^2,xy,y^3)$ is the fiber product $k[x]/(x^2)$ and $k[y]/(y^3)$ over $k$. In view of Example \ref{2113}
reflexive modules over $R_1$ are free.

iii) It is shown in \cite{saeed1} that
the ring  presented in Observation \ref{m3COR} is quasi-decomposable.
In view of Observation \ref{m3COR}(iii) we can find a reflexive module which is not free.
\end{proof}

\subsection{Freeness  of $M^\ast$ versus freeness of $M$}
Recall from Ramras' work that over any local ring of depth zero, freeness  of $M$ follows from  freeness of $M^\ast$.

\begin{observation}
	Let $(R,\fm)$ be a local ring such that freeness of each  module $M$ follows from freeness  of its dual. Then 	$\depth(R)=0$. 
\end{observation}

\begin{proof}	Suppose on the contradiction that $\depth(R)\geq 1$.
	We look at $M:=k\oplus R$. Then $M^\ast$ is free and $M$ is not free.
	We conclude from this contradiction that
	$\depth(R)=0$. 
\end{proof}

The following is converse to \cite{dao}:
\begin{proposition}
	Let $(R,\fm,k)$ be a local ring. Suppose freeness of each torsion-less module $M$ follows from freeness  of its dual. 
Then $\depth(R)<2$.
\end{proposition}

\begin{proof} 
Suppose on the contradiction that $\depth(R)\geq 2$.
	We look at $0\to\fm\to R\to k\to 0$. This induces the following long exact sequence
	$$0\lo\Hom(k,R)\lo \fm^\ast\lo R^\ast\lo\Ext^1_R(k,R)\lo 0.$$
	Since $\depth(R)>1$, we have $\Ext^0_R(k,R)=\Ext^1_R(k,R)=0$.
	From this,
	$\fm^\ast\cong R$ which is free. By the assumption, $\fm$ is free.
	Thus
	$$2\leq\depth(R)\leq\dim(R)\leq\gd(R)=\pd(k)\leq\pd(\fm)+1=1.$$ We conclude from this contradiction that
	$\depth(R)<2$.
\end{proof}

\subsection{Freeness  of certain  reflexive  modules}
Here, modules are finitely generated. When we proved the following result, we were not aware of \cite[Proposition 5.2]{bass2}. Therefore, we present a modern proof of it:
\begin{proposition}\label{msta}(After Bass)
	Let $(R,\fm,k)$ be a local ring. Suppose any module of the from $M^\ast$ is free. Then  $(R,\fm)$ is a regular ring of dimension at most two. 
\end{proposition}

\begin{proof}Firstly, we assume that $\depth(R)=0$. Let $M$ be finitely generated. By our assumption, $M^\ast$ is free. Recall that Ramras proved that freeness  of $M^\ast$ implies freeness of $M$.
	We deduce from this that any finitely generated module is free. In particular, $\frac{R}{\fm}$
	is free which implies that $\fm=0$ and so $R$ is a field.
Now, we assume that $\depth(R)=1$.
	 We look at $0\to\fm\to R\to k\to 0$. This induces the following  exact sequence
	$$0=\Hom(k,R)\lo \fm^\ast\lo R^\ast\lo\Ext^1_R(k,R)\lo 0.$$
		Since $\depth(R)=1$ we know that $\Ext^1_R(k,R)\cong\oplus k $ is  nonzero.
	In particular, $\pd(k)\leq 1$. This implies that  $1=\depth(R)\leq\gd(R)=1.$
	Thus, $R$ is a principal ideal domain.
	Finally, we assume that  $d:=\depth(R)\geq 2$.
	We look at the exact sequence$$0\lo\Syz_{d}(k)\lo R^n\lo \Syz_{d-1}(k)\lo 0.$$ This induces the following long exact sequence
	$$0\lo\Hom(\Syz_{d-1}(k),R)\lo R^n\lo \Hom(\Syz_{d}(k),R)\lo\Ext^1_R(\Syz_{d-1}(k),R)\lo 0,$$
	and that $\Ext^1_R(\Syz_{d-1}(k),R)\cong \Ext^d_R(k,R)$. 
	Since $\depth(R)=d$ we know that $\Ext^d_R(k,R)\cong\bigoplus_{\neq\emptyset} k $ is  nonzero.
By our assumption, $\Syz_{d-1}(k)^\ast$ and $\Syz_{d}(k)^\ast$ are free.	We deduce
	that $\pd(k)\leq 2$. This implies  $R$ is a regular ring of dimension $2$. 
\end{proof}

The following  item suggested by Souvik Dey:

\textbf{Second proof of Proposition \ref{msta}.}
Replace totally reflexive with free in the proof of Proposition \ref{mstac}. We leave the routine modification to the reader.  $\Box$\\

Over normal domains, freeness of rank-one reflexive modules implies $\UFD$. Over complete normal rings, freeness of reflexive modules of rank at most two  implies regularity, see \cite[2.14]{t}.

\begin{proposition}\label{4.1}Let $(R,\fm)$ be a local quasi-reduced ring. Then   reflexive modules are free if and only if $R$ is a regular ring of dimension at most two.
\end{proposition}

\begin{proof}
Over quasi-reduced rings, second syzygy modules  are reflexive modules  (see \cite[page 5809]{mas}). Trivially, second syzygy modules are free if and only if $R$ is regular and of dimension at most two.
\end{proof}

\begin{remark}The first item shows that  the local assumption is important. The second
	item shows that quasi-reduced assumption is needed.   
 \begin{enumerate}
	\item[i)] 
 Consider the ring  $R:=\mathbb{Z}/2\mathbb{Z}\oplus \mathbb{Z}/2\mathbb{Z}$ and $M:=\mathbb{Z}/2\mathbb{Z}\oplus 0$. It is projective and so reflexive. However, $M$ is not free. 	In order to find an example which is integral domain, we look at
	 $R:=\mathbb{Z}[\sqrt{5}]$ and $M:=(2,1+\sqrt{5})\lhd R$.
\item[ii)] See Observation 4.1.
\end{enumerate}
\end{remark}
One has $M^\ast=\Syz_2(D(M))$. Then $M$ is reflexive if $M\cong D_2(D_2(M))$ where $D_2(M):=\Syz_2(D(M))$. Following Auslander-Bridger, this is equivalent  to saying
that the map $$\Ext^1_R(D_2(D_2(M)),-)\stackrel{\cong}\lo \Ext^1_R(M,-)$$ induced by $M\to D_2(D_2(M))$ is an isomorphism.	 \begin{definition}Set $D_n(-):=\Syz_n(D(-))$. \begin{enumerate}
		\item[i)]
Following Auslander-Bridger, $M$ is called $n$-reflexive if  $\Ext^1(D_n(D_n(M)),-)\stackrel{\cong}\lo \Ext^1(M,-)$ induced by $M\to D_n(D_n(M))$ is an isomorphism.	\item[ii)] Following Masek, $R$ is called $n$-Gorenstein if  it satisfies $(\Se_n)$ and is Gorenstein in codimension $n-1$.\end{enumerate}
\end{definition}

\begin{proposition}
	Let $(R,\fm,k)$  be an
	$(n-1)$-Gorenstein  local ring. Then   $n$-reflexive modules are free if and only if $R$ is a regular ring of dimension at most $n$.
\end{proposition}

\begin{proof}
	First, we assume that $n$-reflexive modules are free.
	Let $M:=\Syz_n(k)$. By definition, this is $n$-syzygy.
	Over $(n-1)$-Gorenstein   rings, $n$-syzygy modules  are $n$-torsionless (see \cite[Corollary 43]{mas}). In view of \cite[Theorem 2.17]{ABr} we observe that any $n$-torsionless module
	is $n$-reflexive. So, $M$ is $n$-reflexive. By the assumption, $M$ is free.
	From this, $\pd(k)\leq n$. Thus, $R$ is a regular ring of dimension at most $n$. 
	To see the converse part, let $R$ be regular of dimension at most $n$ and $M$ be $n$-reflexive.
	Recall from \cite[Corollary 4.22]{ABr} that over Gorenstein rings $n$-reflexivity coincides with $n$-syzygy.
	Since $R$ is regular  of dimension at most $n$ any $n$-syzygy module is free. From this,
	$M$ is free. 
\end{proof}

\begin{proposition}\label{mstac}
	Let $(R,\fm,k)$ be a local ring. Suppose any module of the from $M^\ast$ is totally reflexive. Then  $R$ is a Gorenstein ring of dimension at most two. 
\end{proposition}

\begin{proof} 
Recall that $M^\ast$ is isomorphic  to the second syzygy module of $\D(M)$. Due to the assumption, we know that $\Syz_2(\D(  M))$ is totally reflexive for every module $M$. 
In particular, $\Syz_2(\D(\D(  M)))$  is  totally reflexive  for every $R$-module $M$.
Since $\D (\D (M))\approx M$, $\Syz_2 (M)$ is  totally reflexive  for every $R$-module $M$. So, every $R$-module has G-dimension at most $2$. Therefore, $R$ is a Gorenstein ring of dimension at most $2$.
\end{proof}

\section{Reflexivity of multi-duals}

For simplicity, we set $M^{\ell\ast}:=M^{\overbrace{{\ast\ldots\ast}}^{\ell-times}}$ and we put $M^{\ell\ast}:=M$ if $\ell=0$. Reflexivity of $M^\ast$ is subject of \cite[1.3.6]{cam}.
Over any
non-Gorenstein artinian local ring, dual of the residue field is not reflexive. Despite of this, and as a
motivation, we recall  the following result of Vasconcelos:

\begin{fact}\label{5.1v}
Over quasi-normal rings, $M^\ast$ is reflexive. In particular, $M^{\ell\ast}$ is reflexive for all $\ell\in \mathbb{N}$.
\end{fact}

The following extends \cite[Page 518]{Lam} where Lam worked with $\frac{k[x,y]}{(x,y)^2}$ and $M:=k$.

\begin{corollary} \label{ca} Let  $0\neq M$ be any nonfree
over one of the following local rings:
 \begin{enumerate}
\item[i)]  $R:=\frac{k[X_1,\ldots,X_m]}{(X_1,\ldots,X_m)^n}$ with $m>1$, or
\item[ii)] $R$ is such that $\fm^2=0$ and $\mu(\fm)>1$, or
\item[iii)] $R:=\frac{\mathbb{Q}[X,Y]}{(X^2,XY,Y^3)}$.
\end{enumerate}
Then $M^{\ell\ast}$ is not reflexive
  for all $\ell\in \mathbb{N}_0$.
  \end{corollary}

  \begin{proof}We may assume $\ell\in \mathbb{N}$.
  We argue by induction on $\ell$. Without loss of the generality we may assume that $\ell=1$.
   \begin{enumerate}
\item[Claim]  A) Let $(S,\fn)$ be an artinian  local ring and $N\neq 0$ be  finitely generated. Then $N^{\ast}$ is nonzero. Indeed,
since $N$ is finitely generated and nonzero, $$\Ass(\Hom(N,S))=\Supp(N)\cap\Ass(S)=\{\fn\}.$$ In particular, $N^{\ast}\neq 0$.  \end{enumerate}

In view of Claim A),  we see that $M^{\ast}$ is  nonzero. One may find easily that $M^{\ast}$ is finitely generated.
  Suppose on the contradiction that $M^{\ast}$ is  reflexive. The ring $R$ is $\BNSI$ (see \S 4). By Fact \ref{www}, $M^{\ast}$ is free. Since $0\leq\depth(R)\leq \dim(R)=0$ we have  $\depth(R)=0$.
 If depth of a ring is zero, then
 freeness descents from $(-)^\ast$
to $(-)$, see    \cite[Lemma 2.6]{ram}.
This immediately implies that $M$  is free which is excluded from the assumption. This is a contradiction.
  \end{proof}

The above proof shows:

\begin{observation}  Let $(R,\fm)$ be  artinian $\BNSI$  and $M$ be nonfree. Then  $M^{\ell\ast}$ is  not reflexive.
\end{observation}
  It may be natural to ask:

\begin{question}
Let $R$ be  artinian non-Gorenstein  and $M$ be nonfree. When is  $M^{\ell\ast}$  (non-)reflexive?
\end{question}

	Here, is the answer:

	\begin{proposition}\label{gf2}
  Let $(R,\fm)$ be artinian and $M$ be finitely generated. If  $M^{\ell\ast}$  is reflexive
  for some $\ell>1$, then $M^{n\ast}$  is reflexive
  for all $n>0$.
\end{proposition}

\begin{proof}
First, assume $n\geq \ell$. Taking $(n-\ell)^{th}$ times dual from $M^{\ell\ast}\stackrel{\cong}\lo M^{(\ell+2)\ast}$
yields that  $M^{n\ast}\stackrel{\cong}\lo M^{(n+2)\ast}$. From this,
$M^{n\ast}$  is reflexive. Now, assume that $n< \ell$. We set $N:=M^{(\ell-2)\ast}$. Then $N^{2\ast}$ is reflexive. We are going to show that $N^\ast$ is reflexive. 
By definition, the natural map $\phi:=\phi_{N^{2\ast}}:N^{2\ast}\lo N^{4\ast}$
is an isomorphism. This gives a map $\varphi:=N^{4\ast}\lo N^{2\ast}$
such that $\varphi\phi=\id_{N^{2\ast}}.$ 
Since dual modules are torsionless,
there is an exact sequence $$0\lo N^\ast\stackrel{\phi_1}\lo N^{3\ast}\lo C\lo 0,$$
where $\phi_1:=\phi_{N^\ast}$. Taking dual, it yields that $$0\lo C^\ast\lo N^{4\ast}\stackrel{\phi^\ast}\lo N^{2\ast}.$$It is easy to see $\phi^\ast=\varphi$. In particular,
$C^\ast=\ker(\phi^\ast)=0$. Thus,
$$\emptyset=\Ass(C^\ast)=\Ass(\Hom(C,R))=\Supp(C)\cap \Ass(R)=\Supp(C)\cap\{\fm\},$$this is equivalent
to saying that $\fm\notin \Supp(C)$, i.e., $C=0$. Consequently,  $N^\ast\stackrel{\phi_1}\lo N^{3\ast}$ is an isomorphism. So, $M^{n\ast}=N^\ast$  is reflexive,
as claimed.
\end{proof}

 Suppose the ring is artinian and $M$ is nonfree. It is  easy to see that $\limsup \frac{\ell(M^{n\ast})}{\type(R)^{n}}\leq \ell(M).$
 Indeed, via induction on $n$, it is  enough to show $\ell(M^{\ast})\leq \type(R)  \ell(M)$.  To show this it is enough to use induction on $\ell(M)$, or see Menzin's PhD-thesis.
If the module is simple or $R$ is Gorenstein, then $\lim_{n\to\infty} \frac{\ell(M^{n\ast})}{\type(R)^{n}}=\ell(M)$. We ask:

 \begin{question}
 When does the limit $\lim_{n\to\infty} \frac{\ell(M^{n\ast})}{\type(R)^{n}}$ exist? when is it equal to $\ell(M)$?
 \end{question}

 In general, this is not the case:

 \begin{example}
Let $(R,\fm)$ be    artinian  non-Gorenstein and $M$ be free. Then $\lim_{n\to\infty} \frac{\ell(M^{n\ast})}{\type(R)^{n}}=0.$
\end{example}

\begin{proof}
Recall that $M^{n\ast}=M$ for all $n$, since $M$ is free. Also, $\type(R)>1$ because $R$ is  not Gorenstein. Thus, $\lim_{n\to\infty} \frac{\ell(M^{n\ast})}{\type(R)^{n}}=\lim_{n\to\infty} \frac{\ell(M)}{\type(R)^{n}}=0.$
\end{proof}

In particular,
we consider to modules with no free
direct summands:

\begin{example}
Let $R:=\frac{\mathbb{Q}[X,Y]}{(X^2,XY,Y^3)}$ and $M:=\fm$. Then $\lim_{n\to\infty} \frac{\ell(M^{n\ast})}{\type(R)^{n}}\neq\ell(M)$.
\end{example}

\begin{proof}
Note that $R=\mathbb{Q}\oplus \mathbb{Q}x\oplus \mathbb{Q}y\oplus \mathbb{Q}y^2$ and $(0:\fm)=(x,y^2)$. From this,  type of $R$ is $2$. Also, $\ell(\fm)=3$. Since $xy=0$, we have
$$\fm=(x,y)=xR\oplus yR=R/(0:x)\oplus R/(0:y)=R/(x,y)\oplus R/(x,y^2)\quad(\ast)$$Now, we compute dual of $R/(x,y^2)$:$$\Hom_R(R/(x,y^2),R)= \{r:r(x,y^2)=0\}=(x,y)=\fm\quad(\ast,\ast)$$
We combine $(\ast)$ along with $(\ast,\ast)$ to see that
 $\fm^\ast=\mathbb{Q}^\ast\oplus\fm$.  We take another dual to see $$\fm^{\ast\ast}=\mathbb{Q}^{\ast\ast}\oplus\fm^\ast=\mathbb{Q}^{\ast\ast}\oplus \mathbb{Q}^\ast\oplus\fm.$$ By   induction,
$$\fm^{n\ast}=\mathbb{Q}^{n\ast}\oplus \mathbb{Q}^{(n-1)\ast}\oplus\ldots\oplus\mathbb{Q}^\ast\oplus\fm\quad(\ast\ast\ast)$$Recall that $\mathbb{Q}^\ast=\oplus_{\type(R)} \mathbb{Q}$ and that $\mathbb{Q}^{n\ast}=\oplus_{\type(R)^n} \mathbb{Q}=\mathbb{Q}^{2^n}$. We put this along with $(\ast\ast\ast)$ to see
$$\ell(\fm^{n\ast})=\sum_{j=1}^n 2^j+\ell(\fm)=\sum_{j=0}^n 2^j+2=\frac{1-2^{n+1}}{1-2}+2=2^{n+1}+1.$$ Consequently,
$\lim_{n\to\infty} \frac{\ell(M^{n\ast})}{\type(R)^{n}}=\lim_{n\to\infty} \frac{2^{n+1}+1}{2^{n}}=2<3=\ell(\fm).$
\end{proof}

\begin{example}
Let $R:=\frac{\mathbb{Q}[X,Y]}{(X^2,XY,Y^2)}$ and $M:=\fm$. Then $\lim_{n\to\infty} \frac{\ell(M^{n\ast})}{\type(R)^{n}}=\ell(M)$.
\end{example}

\begin{proof}
Note that $R=\mathbb{Q}\oplus \mathbb{Q}x\oplus \mathbb{Q}y$ and $(0:\fm)=(x,y)$. From this, type of $R$ is $2$. Since $xy=0$, we have
$$\fm=(x,y)=xR\oplus yR=R/(0:x)\oplus R/(0:y)=R/(x,y)\oplus R/(x,y)=\mathbb{Q}\oplus \mathbb{Q}.$$By an easy induction,
 $\fm^{n\ast}=\mathbb{Q}^{n\ast}\oplus \mathbb{Q}^{n\ast}$. Recall that $\mathbb{Q}^\ast=\oplus_{\type(R)} \mathbb{Q}$ and that $\mathbb{Q}^{n\ast}=\mathbb{Q}^{2^n}$.
This means that $\ell(\fm^{n\ast})=2^{n+1}.$ Consequently,
$$\lim_{n\to\infty} \frac{\ell(M^{n\ast})}{\type(R)^{n}}=\lim_{n\to\infty} \frac{2^{n+1}}{2^{n}}=2=\ell(\fm),$$as claimed.
\end{proof}

\begin{proposition}\label{lim2m}
Let  $(R,\fm)$ be such that $\fm^2=0$ and $M$ be a finitely generated module with no free
direct summands. Then $\lim_{n\to\infty} \frac{\ell(M^{n\ast})}{\type(R)^{n}}=\beta_0(M)$.
\end{proposition}

\begin{proof}We may assume that $\fm\neq 0$.
Due to \cite[Proposition 1]{men}, $\ell(M^{\ast})= \ell(\fm)\mu(M)$.
We remark that $M^\ast$ is torsion-less. It is submodule of a free module $F$.
Let $f:M^\ast\hookrightarrow F$. Then $M^\ast$ is a first syzygy  of $\coker(f)$ with respect to
a free resolution of $\coker(f)$. Let $\Syz_1(\coker(f))$ be the first syzygy  of $\coker(f)$ with respect to
the minimal free resolution of $\coker(f)$. There are  free modules $R^n$ and   $R^m$ $(m\leq n)$ such that
$$\Syz_1(\coker(f))\oplus R^m\simeq M^\ast \oplus R^n.$$
 By Krull-Schmidt theorem over complete rings, 
 $$M^\ast\simeq \Syz_1(\coker(f)) \oplus R^{n-m}.$$
We recall from \cite[Lemma 2.6]{ram} that $M^\ast$  has   no free
direct summands, because $M $  has   no free
direct summands and depth of the ring is zero. We conclude that $M^\ast\simeq \Syz_1(\coker(f))$, and so that $M^\ast\subset \fm F$ for some free module $F$. In particular,
$\fm M^{\ast}\subset\fm^2F=0$. Hence $$\mu(M^\ast)=\dim (\frac{M^\ast}{\fm M^{\ast}})=\dim  M^\ast=\ell(M^\ast).$$
By repeating this, $M^{n\ast}$ has   no free
direct summands for all $n\geq1$.  Also,
$\ell(M^{n\ast})=\mu(M^{n\ast})$ for all $n\geq1$. By the mentioned result of Menzin, $\ell(M^{(n+1)\ast})= \ell(\fm)\mu(M^{n\ast})$. An easy induction implies that
$$\ell(M^{(n+1)\ast})= \ell(\fm)^n\mu(M^{\ast}).$$   We use $\fm^2=0$ to deduce $\ell(\fm)=\type(R)$. Therefore,
$$\lim_{n\to\infty} \frac{\ell(M^{(n+1)\ast})}{\type(R)^{n+1}}=\lim_{n\to\infty} \frac{\type(R)^{n}\mu(M^{\ast})}{\type(R)^{n+1}}=\frac{\mu(M^{\ast})}{\type(R)}=\frac{\ell(M^\ast)}{\ell(\fm)}=\mu(M)=\beta_0(M),$$as claimed.
\end{proof}

\begin{corollary}
In addition to Proposition \ref{lim2m} assume that $M$ is torsion-less. Then $$\lim_{n\to\infty} \frac{\ell(M^{n\ast})}{\type(R)^{n}}=\ell(M).$$
\end{corollary}

Let $(R,\fm)$ be non-Gorenstein   such that $\fm^2=0$ and $M$ be a module with no free
direct summands.  Note that $\ell(R)\neq 2$, i.e., $(\ell(R)-1)^2-1\neq 0.$ It follows by  \cite[Proposition 1]{men} that
$$\lim_{i\to\infty}\frac{\ell(\Ext^i_R(M,R))}{\ell(R)^i}=\lim_{i\to\infty}\frac{(\ell(R)-1)^{i-2}\beta_1(M)((\ell(R)-1)^2-1)}{\ell(R)^i}=\beta_1(M).$$

\begin{question}Let $(R,\fm)$ be an artinian non-Gorenstein ring.
When is the limit  $\lim_{i\to\infty}\frac{\ell(\Ext^i_R(M,R))}{\ell(R)^i}$  exist?
\end{question}

The following extends the \textit{Third Dual Theorem}, see \cite[19.38]{Lam} (and it is well-known over abelian groups, see \cite[Ex. 12.11.(3)]{fuchs}):

\begin{proposition}\label{gfu}
Let $(R,\fm)$ be a  local ring and $M$ be weakly reflexive. Then $M^{\ell\ast}$ is reflexive  for all $\ell\in \mathbb{N}$.
\end{proposition}

\begin{proof} In view of the Third Dual Theorem, we need to proof the claim only for $\ell:=1$.
To this end, we take dual from  $M\stackrel{\varphi_M}\lo M^{\ast\ast}\to 0$  to observe that  $0\to M^{\ast\ast\ast}\stackrel{(\varphi_M)^\ast}\lo M^\ast$.
We look at
$\varphi_{M^\ast}:M^\ast\stackrel{}\lo  M^{\ast\ast\ast}$ and realize that $(\varphi_M)^\ast\varphi_{M^\ast}=1_{M^\ast}$. That
 is the monomorphism $\varphi_{M^\ast}:M^\ast\stackrel{}\lo M^{\ast\ast\ast}$ splits. We have $$M^{\ast\ast\ast}\simeq M^\ast\oplus \coker(\varphi_{M^\ast})\simeq
M^\ast\oplus\ker((\varphi_M)^\ast)\simeq M^\ast\oplus 0\simeq M^\ast.$$
By \cite[1.1.9(b)]{kris} $M$ is reflexive if it is isomorphic to its bidual.
\end{proof}

\begin{fact} \label{mas}(Bass 1963, Hartshorne  1992, Masek 2000) Let $(R,\fm)$ be a locally  quasi-reduced  ring. Then
$M^{\ell\ast}$ is reflexive  for all $\ell\in \mathbb{N}$. Conversely, if $M^{\ast}$ is reflexive  for all $M$, then $R$ is quasi-reduced.
\end{fact}

\begin{proof}We may assume that $R$ is local.
	In the sense of Masek, $R$ is 1-Gorenstein. Over a such ring, $2-$torsionless is the same as of $2-$syzygy, see
	\cite[Corollary 45]{mas}.  Let $F_1\to F_0\to M\to 0$ be a minimal presentation of $M$.
	In view of $$0\lo M^{\ast}\lo F_0^{\ast}\lo F_1^{\ast}\lo \D(M)\lo 0 $$ we know $M^\ast$ is second syzygy. Thus, $M^\ast$ is $2-$torsionless, and so reflexive.
From this, $M^{\ell\ast}$ is reflexive  for all $\ell\in \mathbb{N}$.

Conversely, by a result of Bass (see \cite[Proposistion 6.1]{bass}) we know $S^{-1}R$ is a Gorenstein ring where $S$ is the multiplicative closed subset  of all  regular elements.
Note that $S=R\setminus \cup _{\fp\in\Ass(R)}\fp$. Since $S^{-1}R$ is equai-dimensional, $\min(S^{-1}R)=\Ass(S^{-1}R)$.
From this, $\min(R)=\Ass(R)$. This property is just $(\Se_1)$. Let $\fp\in\min(R)$. Since $(S^{-1}R)_{S^{-1}\fp}=R_{\fp}$ we deduce that $R$
is $(G_0)$. Recall that quasi-reduced means $(G_0)$+$(\Se_1)$. This completes the proof.
\end{proof}

\begin{remark} The quasi-reduced assumption is important, see Observation 4.1.  Also,
the finitely generated assumption on $M$ is important even over $\mathbb{Z}$:

 \begin{enumerate}
\item[i)]
There is an abelian group $G$ such that $G^\ast$ is not reflexive, see e.g. \cite[XI, Theorem 1.13]{ek}.
\item[ii)] There is a non-reflexive abelian group $G$ such that $G\simeq G^{\ast\ast}$, see \cite[Page 355]{ek}
(by \cite[Proposition 1.1.9]{kris} such a thing never happens in the setting of noetherian modules).
The origin proof uses topological methods.
\end{enumerate}
\end{remark}
Recall that there are natural maps $M^{m\ast}\to (M^{m\ast} )^{\ast\ast}= M^{(m+2)\ast} $. We set
$$
\begin{CD}
M^{\infty\ast}_{even}:=\vil\big( M@>\varphi_{{M}}>> M^{2\ast} @>\varphi_{{M^{2\ast}}}>>M^{4\ast} @>\varphi_{{M^{4\ast}}}>> M^{6\ast} @>>> \ldots\big)\\
\end{CD} $$ 

 and

$$
\begin{CD}
M^{\infty\ast}_{odd}:=\vil\big( M^{\ast}@>\varphi_{{M^{\ast}}}>> M^{3\ast} @>\varphi_{M^{3\ast}}>>M^{5\ast} @>\varphi_{{M^{5\ast}}}>> M^{7\ast} @>>> \ldots\big).\\
\end{CD} $$

\begin{question}  When is  $M^{\infty\ast}_{odd}\cong M^{\infty\ast}_{even}$? 
\end{question}

The question is not true on this generality:

\begin{example}
	Let $R:=\mathbb{Z}$ and $M:=\bigoplus_{\mathbb{N}}\mathbb{Z}$. Recall that  $M^\ast=\prod_{\mathbb{N}}\mathbb{Z}$ and  $(\prod_{\mathbb{N}}\mathbb{Z})^\ast=\bigoplus_{\mathbb{N}}\mathbb{Z}$. From these, 
	 \begin{enumerate}
		\item $M^{\infty\ast}_{even}=\vil \big( M\to M^{2\ast}\to M^{4\ast}\to M^{6\ast}\to \cdots\big)\cong \bigoplus_{\mathbb{N}}\mathbb{Z}$
		\item $M^{\infty\ast}_{odd}=\vil \big(M^{\ast}\to M^{3\ast}\to M^{5\ast}\to M^{7\ast}\to \cdots\big)\cong\prod_{\mathbb{N}}\mathbb{Z}.$
		 \end{enumerate}
	So, $M^{\infty\ast}_{odd}\ncong M^{\infty\ast}_{even}$. However, both of $M^{\infty\ast}_{odd}$ and $ M^{\infty\ast}_{even}$ are reflexive.
\end{example}

\begin{remark}
Let $R$ be a local domain and $M$ be finitely generated. Then $ M^{\infty\ast}_{odd}\cong M^\ast$ and $ M^{\infty\ast}_{even}\cong M^{\ast\ast}$ are reflexive. In particular, there are situations for which:  \begin{enumerate}
	\item[$\bullet$]$M^{\infty\ast}_{odd}\ncong M^{\infty\ast}_{even}$, and also \item[$\bullet$]$M^{\infty\ast}_{odd}\cong M^{\infty\ast}_{even}$. \end{enumerate}
\end{remark}

\section{Treger's conjecture}
Here, modules are finitely generated.

\begin{example}\label{4.2}Let
$R:=\mathbb{C}[[ x^4,x^3y,x^2y^2,xy^3,y^4]]$. Then  $R$ is a complete normal domain of dimension $2$ and every nonzero
reflexive module  decomposes into a direct sum of rank one submodules.
\end{example}

\begin{proof}
Let $A := \mathbb{C}[[x, y]]$ and recall that $R$ is the $4$-Veronese subring  of $A$.  Recall that $R$ can be regarded as an invariant ring
 of cyclic group $G:=(g)$ of order $4$. Let $A_i$ be the invariant ring by $g^i$. We left to the reader to check
 that each $A_i$ is reflexive as an  $R$-module and of rank one. This is well-known that
$A=A_0\oplus\ldots\oplus A_{3}$. Let $M$ be any indecomposable reflexive $R$-module. By \cite[Proposition 6.2]{lw},
any  indecomposable reflexive $R$-module is a direct summand of $A$ as an $R$-module.
From this, there is an $R$-module $N$ such that $$M\oplus N\cong A=A_0\oplus\ldots\oplus A_{3}.$$ 
Krull-Remak-Schmidt property holds over complete rings.
	It follows   that $M\cong A_{i}$ for some $0\leq i \leq3$. Recall that $A_i$ is  of rank one. 
In particular, any indecomposable reflexive module is of rank one.
\end{proof}
 The following  demonstrates the role of 2-dimensional assumption in  Treger's conjecture.

\begin{observation}
Let $A:=\mathbb{C}[[X_1,\ldots,X_n]]$ and let $(R,\fm)$ be its  $m$-Veronese subring. Any  nonzero reflexive module  decomposes into a direct sum of rank one submodules
if and only if $n\leq2$.
\end{observation}

\begin{proof}
In the case $n=1$, the ring $R$ is regular and is of dimension one. Remark \ref{4.1} shows that reflexive modules are free.
Also, the case $m=1$ is trivial. The case $n=2$ is a modification of Example \ref{4.2} (it may be worth to note that there is a geometric proof  in \cite[Lemma 1]{ile}). Next we deal with
 $n=3$ and $m=2$.  This is stated in \cite[Theorem 4.1]{ar} that
 the canonical module of $R$ is generated by $3$ elements and not less. Recall that
 $\omega_R$ is of rank $1$. However, if we pass to its first syzygy we get a rank $2$
 indecomposable maximal Cohen-Macaulay module $\Syz(\omega_R)$. In particular,
$\Syz(\omega_R)$ is a reflexive module and of rank two. Since $\Syz(\omega_R)$ indecomposable, it does not decomposable into a direct sum of rank one submodules.
Finally we  assume either $n>3$ or ($n=3\leq m)$. Again, we use a result of Auslander and Reiten (\cite[Theorem 3.1]{ar}) to see  there are infinity many
indecomposable maximal Cohen-Macaulay modules. Hence, there are infinity many
indecomposable reflexive modules. The classical group of $R$
is $\mathbb{Z}/m\mathbb{Z}$. By definition,
there are finitely many rank one reflexive modules. Now, we find  an indecomposable reflexive module $M$
of rank bigger than one. In particular, $M$ is not direct sum of its rank one submodules.
\end{proof}

\begin{proposition}\label{plane}Let $(R,\fm)$ be a singular standard graded normal  hypersurface ring of dimension $2$ where $k$
is algebraically closed  with $\Char k\neq2$.  Then $R=\frac{k[x,y,z]}{(x^2+y^2+z^2)}$ (after a suitable linear change of variables) if and only if every nonzero
graded reflexive module  decomposes  into a direct sum of rank one submodules.
\end{proposition}

\begin{proof} Suppose every
graded reflexive module is decomposable into a direct sum of rank one submodules. Set $\mathcal{C}:=\Proj(R)$. In view of  Serre's criterion of normality (\cite[Page 185]{har}),
$\mathcal{C}$ is a smooth projective  plane curve.  According to \cite[Proposition I.7.6]{har},
 $\mathcal{C}$ is of degree equal to $d:=\deg(f)$. Also, $\mathcal{C}$ is reduced, irreducible and connected.
Recall that there is no difference between reflexive modules
and maximal Cohen-Macaulay  modules. Maximal Cohen-Macaulay are locally free over punctured spectrum and their are of constant rank, since $R$ is domain and is regular over punctured spectrum. In this regards, graded reflexive modules correspondence to locally free sheaves. Let us explain a little more. Namely,
an $\mathcal{O}_\mathcal{C}$-module $\mathcal{F}$ is called \textit{free} if it is isomorphic to a direct sum of copies of
$\mathcal{O}_\mathcal{C}$. It is called \textit{locally free} if $\mathcal{C}$ can be covered by open (affine) sets $U$ for which $\mathcal{F}|_U$
is a free $\mathcal{O}_U$-module. In particular, $\mathcal{F}$ is quasi-coherent. We look at the reflexive module $$M:=\bigoplus_{n\in \mathbb{Z}}\HH^0(\mathcal{C},\mathcal{F}\otimes\mathcal{O}_\mathcal{C}(n)),$$see \cite[6.49. Exercise]{lw}. According to \cite[Proposition II.5.15]{har}
$\mathcal{F}\simeq\widetilde{M}$. Here, we use the fact that  $R$ is  finitely generated by $R_1$ as
an $k$-algebra. Recall from \cite[Ex. II.5.18(d)]{har} that there is a one-to-one correspondence
between isomorphism classes of locally free sheaves   and isomorphism
classes of vector bundles.
In sum, we observed that any vector bundle on $\mathcal{C}$ is
isomorphic to a direct sum of line bundles. Again, we are going to use the fact that $k=\overline{k}$: Thanks to \cite[Theorem 1.1]{bal}, $\mathcal{C}$ is isomorphic
to $\PP^1$. In the light of \cite[V.5.6.1]{har} we know
 genus is a birational invariant. 
Since  $\mathcal{C}$ is smooth,  $$g_{\mathcal{C}}=\frac{(d-1)(d-2)}{2}$$ (see \cite[Ex. I.7.2]{har}). This is zero, because $g_{\PP^1}=0$. This implies that $d\leq2$.
It $d=1$ this implies that $R$ is nonsingular which is excluded by the assumption. Then we may assume that $d=2$. Since $\Char k\neq2$ and in view of \cite[Ex. I.5.2]{har} and after a suitable linear change of variables,  $\mathcal{C}$ is defined by $x^2+y^2+z^2=0$.

Conversely, assume $R=\frac{k[x,y,z]}{(x^2+y^2+z^2)}$. In the sense of representation theory, $R$ has singularity of type $A_1$. In particular,  there are only two indecomposable maximal Cohen-Macaulay modules. Both of them are of rank one. The proof is now complete.
\end{proof}

The standard-graded assumption is really important:

\begin{example} i) Any  nonzero reflexive module over $R_n:=\frac{\mathbb{C}[x,y,z]}{(x^n+y^2+z^2)}$  with  $n\geq2$ decomposable into a direct sum of rank one reflexive submodules (see \cite[Example 5.25]{lw}).
In fact any  nonzero indecomposable   reflexive $R_n$-module is of rank $1$. The same thing works for
$\frac{\mathbb{C}[[x,y,z]]}{(x^n+y^2+z^2)}$.

ii) Look at $\mathcal{R}=\frac{k[x,y,z]}{(x^2+y^2+z^2)}$ equipped with graded ring structure given by $G:=\mathbb{Z}/2\mathbb{Z}\oplus\mathbb{Z}/2\mathbb{Z}\oplus\mathbb{Z}/2\mathbb{Z}$. Auslander-Reiten presented an indecomposable $G$-graded reflexive module $M$. For more details, see \cite[Page 191]{AR}.
 Also, they remarked that $\widehat{M}$ decomposes into a direct sum of two rank one reflexive submodules.
\end{example}

\begin{remark}Treger remarked that there are indecomposable reflexive modules over $R=\frac{k[[x,y,z]]}{(x^3+y^5+z^2)}$ of rank bigger than one, see \cite[Remark 3.12]{t} (here we adopt
some restriction on the characteristic). Let us determine the rank of such modules.
\begin{enumerate}
\item[i)]Let $(R,\fm)$ be a complete $2$-dimensional $\UFD$ which is not regular and containing a field. Then any indecomposable reflexive module is of rank $1\leq i\leq 6$.

\item[ii)] Let $(R,\fm)$ be a complete $2$-dimensional ring of Kleinian singularities. Suppose any reflexive module over $R$  decomposes into a direct sum of rank one reflexive submodules.
Then $R\simeq \frac{k[[x,y,z]]}{(x^n+y^2+z^2)}$ for some $n>1$.

\item[iii)] If we allow the ring is not local, then we can construct a  $2$-dimensional $\UFD$
with a projective module which is not free.
\end{enumerate}
\end{remark}

\begin{proof} Since $R$ is $\UFD$, its classical group is trivial. This implies that any rank one reflexive module is free. According to Remark  \ref{4.1},
there is a reflexive module which is not free. By this, there is an indecomposable reflexive module of rank bigger than one.

i)   Lipman proved over any
algebraically closed field $k$ of characteristic $>5$ the only non-regular normal
complete $2$-dimensional local ring which is a $\UFD$ is $R=\frac{k[[x,y,z]]}{(x^3+y^5+z^2)}$ (see e.g. \cite[Ex. V.5.8]{har}).
  According to \cite{dan}
 there is  a complete list of representatives of the isomorphism classes of indecomposable  maximal Cohen-Macaulay modules.
The ranks are  $[1,6]$.

ii) The proof of this is similar  to i) and we leave it to reader.

iii) This item was proved by many authors. For example, the tangent bundle of a 2-dimensional sphere is not trivial, see Swan's monograph \cite{s}.
\end{proof}

\begin{remark}
Let  $R=\frac{k[x_0,\ldots,x_3]}{I}$   be a singular standard graded normal  ring of dimension $2$ where $k=\overline{k}$
with $\Char k\neq2$. Suppose $\deg(R)\leq 7$. Then $I$ can be generated by degree-two elements  if every nonzero
graded reflexive module  decomposes into a direct sum of rank one submodules.
\end{remark}

\begin{proof} Let $\mathcal{C}:=\Proj(R)$. We may assume that $\mathcal{C}$ is genus is zero, see Proposition \ref{plane}. The case $\deg(R)\leq 4$ follows from the following fact:
 \begin{enumerate}
\item[Fact] (Treger-Nagel) Let $X\subset\PP ^n$ be an arithmetically Cohen-Macaulay variety of degree $d$ and codimension $c>1$. Then the defining ideal  is generated by forms
of degree $\lceil \frac{d}{c}\rceil$.
\end{enumerate}For the simplicity we recall the following classification result: Let $X\subset\PP ^n$
 be an irreducible and reduced normal projective subvariety of
dimension one and degree five  which is not
contained in any hyperplane of $\PP ^n$. Then $X$ is one of the following four types:
i) a  curve of genus $6$ in $\PP ^2$,
ii) a curve of genus $2$ in $\PP ^3$,
 iii)  an elliptic curve of degree $5$ in  $\PP ^4$, or
 iv)  a rational normal curve in $\PP ^5$.
These are not isomorphic with $\mathcal{C}$  (we use the presentation of \cite[page 4324]{tir}).
The only projectively normal curve of degree  $6$ in $\PP ^3$  not contained in any plane
are of genus $3$ or $4$ (see \cite[Ex.V.6.6]{har}). The same citation  shows that the only projectively normal curve of degree  $7$  in $\PP ^3$  not contained in any plane
are of genus  $ 5$ or $6$. The proof is now complete.
\end{proof}

\section{Remarks on  a question by Braun}

 \begin{question} (Braun, \cite[Question 16]{braun}) Let $(R,\fm)$ be a  normal  domain and $I\lhd R$ a reflexive ideal with $\id_R( I) <
\infty$. Is $I$ isomorphic to a canonical module?
\end{question}

By \cite[Page 682]{braun}, the only positive evidence we have is when $R$ is also Gorenstein.

 \begin{proposition} \label{7.2} Let $(R,\fm)$ be an analytically normal  domain  of dimension $2$ and $I\lhd R$ be reflexive with $\id_R( I) <
\infty$. Then $I$ isomorphic to a canonical module.
\end{proposition}

 \begin{notation}For an $R$-module $M$, the $i^{th}$ local cohomology of $M$ with
	respect to an ideal $\fa$ is defined by
	$\HH^{i}_{\fa}(M):={\varinjlim}_n\Ext^{i}_{R} (R/\fa^{n},
	M)$.\end{notation}
Let us present the proof of proposition:
\begin{proof}
Due to the Serre's criterion of normality, the ring is $(\Se_2)$. Since $\dim R=2$, we see that $R$ is Cohen-Macaulay.
  In the light of Fact A) in Observation \ref{m3COR},
$\depth I\geq2$. Since $\dim R=2$ we deduce that $I$ is maximal Cohen-Macaulay.
 We may assume that the ring is complete:  Let $$K_R:=\Hom_R(\HH^2_{\fm}(R),E(R/\fm)).$$ Thanks to   \cite[Definition 5.6]{kunz}, $K_R\otimes \widehat{R}\simeq K_{\widehat{R}}$.
Also, recall that if $M \otimes \widehat{R}\simeq N \otimes \widehat{R}$, then $M$ is also isomorphic to $N$ (see \cite[Lemma 5.8]{kunz}).
Since $R$ is complete and Cohen-Macaulay, $K_R=\omega_R$, in the sense that it is maximal
Cohen-Macaulay module  of type one and of finite injective dimension. For the simplicity, we bring the following from  \cite[Proposition 11.7]{lw}:
 \begin{enumerate}
\item[Fact A)] Let $A$ be a Cohen-Macaulay local ring with canonical module
$\omega_A$. If a module $M$ is both Maximal Cohen-Macaulay and
of finite injective dimension, then $M\simeq \oplus_n\omega_A$ for some $n$.\end{enumerate}
We apply the above fact to see $I\simeq\oplus_n\omega_R$ for some $n$. Since the ring is domain,
$I_1 I_2\neq0$ where   $0\neq I_i\subset I$ are ideals of $R$. Thus, $n=1$ and that $I\simeq\omega_R$.
\end{proof}

 \begin{proposition}  Let $(R,\fm)$ be an analytically normal  domain  and $I\lhd R$ be totally reflexive with $\id_R( I) <
\infty$. Then $I$ isomorphic to a canonical module.
\end{proposition}

\begin{proof}Due to Bass' conjecture (which is a theorem), $R$ is Cohen-Macaulay because there is a finitely generated module of finite injective dimension.
Due to the above argument, we may assume  $R$ is  complete. By definition, totally reflexive are of zero $G$-dimension.
According to  Auslander-Bridger formula, $$\depth(I)=\Gdim(I)+\depth(I)=\depth(R)=\dim(R).$$ In other words,  $I$ is maximal Cohen-Macaulay.
We apply  Fact A) in Proposition \ref{7.2} to see $I\simeq\oplus_n\omega_R$ for some $n$. Since the ring is domain, $I\simeq\omega_R$.
\end{proof}

\begin{remark}
	We recently find an affirmative answer to the Bruan's question in the full setting, see \cite{ufd}.
\end{remark}

\section{Descent from (to)  the endomorphism ring and tensor products}

A module $M$ satisfied $(\Se_n)$ if
$M_{\fp}$ is either maximal Cohen-Macaulay or $0$ when $\Ht(\fp)\leq n$, and
$\depth(M_{\fp})\geq n$ if  $\Ht(\fp)>n$ with the convenience that $\depth(0)=\infty$.
In \cite[Proposition 4.6]{ag}, there is a criterion of reflexivity over normal local domains. Here, we prove this  for Gorenstein rings:

\begin{proposition}\label{decref}
Let $(R,\fm)$ be a Gorenstein ring. The following holds:
\begin{enumerate}
\item[i)] Assume   $\Ext^1_{R}(M,M)=0$ and  $\Hom_{R}(M,M)$ is reflexive. Then  $M$  is  reflexive.
\item[ii)] If $M$ is  reflexive, then $\Hom_{R}(M,M)$  is reflexive.
\end{enumerate}
\end{proposition}

\begin{proof}One can assume
that the ring is  local and $M$ is nonzero.

i):
We argue by induction on $d:=\dim R$. When $d=0$,  any module is totally reflexive. Suppose, inductively, that $d > 0$ and that the result has been proved
for  Gorenstein rings of smaller  dimensions. By the inductive assumption, $M$ is reflexive over the punctured spectrum.
Since $\Hom_{R}(M,M)$ is torsion-less, $$\fm\notin\Ass(\Hom_{R}(M,M))=\Supp(M)\cap \Ass(M)=\Ass(M).$$
If $d=1$, this means that $M$ is maximal Cohen-Macaulay. Since
these modules over Gorenstein rings are reflexive, we can assume that $d>1$. It follows
from definition that $M$ is $(\Se_1)$, i.e., $M$ is torsion-less. We conclude that there is an exact sequence $0\to M\to M^{\ast\ast}\to L\to 0$.
Since $M$ is reflexive over the punctured spectrum, $L$ is of finite length. In particular, $$\Ass(L)\subset\Supp(L)\subset \{\fm\}\subseteq\Supp(M).$$
We are going to show that $L=0$. Keep in mind that a module $(-)$ is zero if $\Ass(-)=\emptyset$. Since $$\Ass(\Hom(M,L))=\Supp(M)\cap\Ass(L)=\Ass(L),$$ we see  $L=0$ if and only if $\Hom(M,L)=0$.
There is an exact sequence $$0\to \Hom_{R}(M,M)\to\Hom_{R}(M, M^{\ast\ast})\to \Hom(M,L)\to \Ext^1_{R}(M,M)=0.$$
Since $\Hom_{R}(M,M)$ is reflexive, $\depth(\Hom_{R}(M,M))\geq2$. In the same vein, $\depth(M^{\ast\ast})\geq2$. In view of \cite[Proposition 4.7]{ag} we observe that
$$\depth(\Hom_{R}(M, M^{\ast\ast}))\geq2.$$ Suppose on the way of contradiction that $\Hom(M,L)\neq 0$. Since $\Hom(M,L)$ is artinian, $\depth(\Hom(M,L))=0$. We use the \textit{depth lemma} to see that $$\depth(\Hom(M,L))\geq\big\{\depth(\Hom_{R}(M,M))-1,\depth(\Hom_{R}(M, M^{\ast\ast}))\big\}\geq 1.$$
This  contradiction shows that $M$ is  reflexive.

 ii):  Recall that $M$ is $(\Se_2)$. Let $\fp\in\Supp(\Hom_{R}(M,M))$ of height at most two. Then $\fp\in\Supp(M)$.
It follows from $(\Se_2)$ that $M_{\fp}$ is maximal Cohen-Macaulay. Due to \cite[Proposition 4.7]{ag},
$$\depth(\Hom_{R_{\fp}}(M_{\fp},M_{\fp}))\geq\Ht(\fp).$$ This means that $ \Hom_{R}(M,M)_{\fp}$ is  maximal Cohen-Macaulay. Let $\fp$ be of height greater than two.
It follows from $(\Se_2)$ that $\depth(M_{\fp})\geq 2$. Another use of \cite[Proposition 4.7]{ag} implies that $\depth(\Hom_{R_{\fp}}(M_{\fp},M_{\fp}))\geq2$.
This means that $\Hom_{R}(M,M)$ is $(\Se_2)$. This condition over Gorenstein rings implies the reflexivity.
\end{proof}

\begin{example}Proposition \ref{decref} (ii) is true over rings that are Gorenstein in codimension one. The Gorenstein assumption in Proposition \ref{decref} (i) is important:
Let $R:=k[[X^3,X^4,X^5]]$. Recall that
for any maximal Cohen-Macaulay module $M$, $\Ext^+_R(M,\omega_R)=0$. Applying this for the canonical module,  $\Ext^+_{R}(\omega_R,\omega_R)=0$.
Also, $\Hom_{R}(\omega_R,\omega_R)\simeq R$  is reflexive.
By \cite[4.8]{tensor}, $\omega_R$ is not reflexive.
\end{example}

The ring  in the next result is more general than \cite[Proposition 4.9]{ag}:

\begin{lemma}\label{vext}
Let $(R,\fm)$ be local, $M$ be  locally free on the punctured spectrum  of depth at least two such that
$\depth (\Hom_{R}(M,M))\geq 3$. Then   $\Ext^1_{R}(M,M)=0$.
\end{lemma}

\begin{proof} We have $\Ext^1_{R}(M,M) $ is of finite length.
Suppose on the contradiction that $\Ext^1_{R}(M,M)\neq0$. Then $\depth(\Ext^1_{R}(M,M))=0$.
Let $x$ be $M$-regular and let $$C:=\ker\bigg(x:\Ext^1_{R}(M,M)\to\Ext^1_{R}(M,M)\bigg).$$ We look at the exact sequence
$$0\to \Hom_R(M,M)/x\Hom_R(M,M) \to\Hom_R(M, M/xM )\to C\to 0.$$Note that \begin{enumerate}
	\item[$\bullet$] $\depth(\frac{\Hom_R(M,M)}{x\Hom_R(M,M)})\geq2$, and
	\item[$\bullet$]  $\depth(\Hom_R(M,\frac{M}{xM}))>0$.
\end{enumerate} 
We use depth lemma to see
$$\depth(C)\geq\left\{\depth\left(\frac{\Hom_R(M,M)}{x\Hom_R(M,M)}\right)-1,\depth\left(\Hom_R(M,\frac{M}{xM})\right)\right\}\geq 1,$$
a contradiction.\end{proof}

The following deals with the reflexivity assumption of \cite[Theorem 2.3]{ces} and may regard
as a generalization of \cite[Theorem 4.4]{ag}:

\begin{corollary}\label{decproc1}
Let $(R,\fm)$ be an abstract complete intersection of dimension at least
$4$. Suppose $M$ is  locally free on the punctured spectrum and
$\depth (\Hom_{R}(M,M))\geq 4$. The following are equivalent: \begin{enumerate}
\item[i)]  $M$ is free,\item[ii)]    $M$ is reflexive,\item[iii)]$\Ext^1_{R}(M,M)=0$.
\end{enumerate}
\end{corollary}

\begin{proof}
 $i)\Rightarrow ii)$ is trivial and  $ii)\Rightarrow iii)$  is a special case of  Lemma \ref{vext}.

 $iii)\Rightarrow i)$
 Since $\Hom_{R}(M,M)$ is $(S_4)$, it is particularly $(S_2)$.
This condition over Gorenstein rings implies the reflexivity. We are in a position to apply Proposition \ref{decref}(i). Due to Proposition \ref{decref}(i) $M$ is reflexive.
This allow us to use \cite[Theorem 2.3]{ces} to show  $M$ is free.
\end{proof}

\begin{fact}\label{vasthm}(See \cite[Theorem 3.1]{wol}) Let $(R,\fm)$ be a one-dimensional Gorenstein ring and $M$ a
finitely generated $R$-module. Then $M$  is projective provided $\Hom_{R}(M,M)$ is  projective.
\end{fact}

Due to an example of Vasconcelos,  Fact \ref{vasthm} can't be extended to higher-dimensional Gorenstein rings.
However, we show:

\begin{proposition}\label{decpro}
Let $(R,\fm)$ be a $d$-dimensional Gorenstein ring. Assume the following conditions:
\begin{enumerate}
\item[i)] $\Ext^i_{R}(M,M)=0$  for all $0<i<d$,\footnote{there is nothing when $d=1$.}
\item[ii)] $\Hom_{R}(M,M)$ is  projective.
\end{enumerate}
Then $M$  is projective.
\end{proposition}

\begin{proof}One can assume
that the ring is  local.
We argue by induction on $d$. When $d\leq1$, the claim is in Fact \ref{vasthm}. Suppose, inductively, that $d > 1$ and that the result has been proved
for Gorenstein  rings of smaller  dimensions. Recall that  $$\fm\notin\Ass(\Hom_{R}(M,M))=\Supp(M)\cap \Ass(M)=\Ass(M).$$ Let $x$ be a regular element
both on $R$ and on $M$. We set $\overline{(-)}:=(-)/x(-)$ and we look at $$0\lo M\stackrel{x}\lo M\lo \overline{M} \lo 0 \quad(\ast).$$ Since $\Ext^1_{R}(M,M)=0$,
 $\Hom_{R}(M,\overline{M})\simeq\overline{\Hom(M,M)}$. We combine this along with
$\Hom_{R}(M,\overline{M})\simeq\Hom_{\overline{R}}(\overline{M},\overline{M})$ to deduce that $\overline{\Hom(M,M)}$ is  free over $\overline{R}$.
First, we may assume $d=2$. Thanks to Fact \ref{vasthm}, $\overline{M}$ is free over $\overline{R}$. The exact sequence
$$0\lo \Syz(M)\lo F\lo M\lo 0,$$induces $$\Tor^{R}_1(M,\overline{R})\lo \overline{\Syz(M)}\lo \overline{F}\lo \overline{M}\lo 0.$$ Since $\Tor^{R}_1(M,\overline{R})=\ker(x:M\to M)=0$ and $\overline{F}\simeq \overline{M}$,
we see $\overline{\Syz(M)}=0$.   Nakayama's lemma says that $\Syz(M)=0$.
By definition, $M$  is free.
For simplicity, we may assume $d=3$.  Again, $(\ast)$
induces the  exact sequence $$0=\Ext^1_R(M,M)\lo\Ext^1_R(M,\overline{M})\lo \Ext^2_R(M,M)=0.$$ Hence
$\Ext^1_R(M,\overline{M})=0$. On the other hand, we have  $\Ext^1_R(M,\overline{M})\simeq\Ext^1_{\overline{R}}(\overline{M},\overline{M})$, and consequently,  $\Ext^1_{\overline{R}}(\overline{M},\overline{M})=0$.
Due to the $2$-dimensional case, $\overline{M}$ is free over $\overline{R}$, and so $M$  is free.
\end{proof}

\begin{example}
 The Gorenstein assumption in Proposition \ref{decpro}  is important. Indeed, let $(R,\fm)$ be any  local complete Cohen-Macaulay ring which is not Gorenstein.
We look at the canonical module.  Recall that $\Ext^+_{R}(\omega_R,\omega_R)=0$ and $\Hom_{R}(\omega_R,\omega_R)\simeq R$. But, $\omega_R$ is not projective.
\end{example}

The following may be extend \cite[Theorem 3.1]{tensor} by weakening of $(\Se_{n})$ to  $(\Se_{n-1})$.

\begin{proposition}\label{n-1}
Let $(R,\fm)$ be a hypersurface of dimension $n>2$  and $0\neq M$ be of constant rank. If $M\otimes M$ is $(\Se_{n-1})$,
then $M$ is free.
\end{proposition}

\begin{proof}Note that $\Supp(M)=\Spec(R)$. From this, $\Supp(M\otimes M)=\Spec(R)$. In the light of \cite[Theorem 3.1]{tensor} we need to show $\depth(M\otimes M)=n$. Suppose on the contradiction that $\depth(M\otimes M)\neq n$.
Since $M\otimes M$ is $(\Se_{n-1})$, we deduce that $\depth(M\otimes M)=n-1$.
Since $n>2$, $M\otimes M$ is $(\Se_2)$. This condition over complete-intersection rings implies the reflexivity.
Due to the reflexivity we are in a position to use Second Rigidity Theorem \cite[Theorem 2.7]{tensor} to see that  $\Tor_i^+ (M,M)=0$. This vanishing result allow us to apply the depth formula (\cite[Proposition 2.5]{tensor}):
$$\depth(M) + \depth(M)
= \depth(R) + \depth(M \otimes M)
=n+(n-1).$$ The left hand side is an even number and the right hand side is odd. This is a contradiction.
\end{proof}

\begin{remark}The first item shows that having the constant rank in Proposition \ref{n-1} is really needed. The second item shows that the assumption  $(\Se_{n-1})$ can't be weekend to  $(\Se_{n-2})$:
	\begin{enumerate}
		\item[i)]  We look at $R:=\mathbb{Q}[X,Y,Z,W]/(XY)$ and $M:=R/ xR $. It is easy to see
		that  $M^{\ell\otimes}$ is $(\Se_{2})$ (resp. reflexive) for all $\ell>0$ but $M$ is not free. Also, this shows that \cite[Theorem 3.1]{miller2}
		needs the extra assumption: the module $M$ has constant rank.
		\item[ii)]    Let $(R,\fm,k)$ be a $3$-dimensional regular local
		ring and let $M:=\Syz_2(k)$. We left to the reader to check that 
		$M\otimes_RM$ satisfies $(\Se_{1})$. Clearly, $M$ is not free.
	\end{enumerate}
\end{remark}

\section{Reflexivity and UFD}

As an application to Corollary \ref{decproc1}, we recover the following:

\begin{theorem}\label{groth}(Grothendieck 1961)
	Let $(R,\fm)$ be a local  complete intersection domain. If $R_P$ is $\UFD$ for all $P$ of height $\leq3$, then $R$ is $\UFD$.
\end{theorem}

\begin{proof}
	The proof is by induction on $d:=\dim R$. We may assume that $d>3$.
	First, we deal with the case $d=4$. Let $\fp$ be any height one prime ideal. 
	Since $\dim R=4$ and $R_{P}$ is $\UFD$ for every prime ideal  $P$ of height $\leq3$ we deduce that $\fp$ is locally free on the punctured spectrum.
	One dimensional normal rings are regular. Thus, $R$ is $(\Se_{2})$ and $(\R_1)$. According to Serre, $R$ is normal. From this, $\Hom_R(\fp,\fp)=R$. In particular,
	$\depth_R(\Hom_R(\fp,\fp))\geq 4$.  Then, $$\fp^{\ast\ast}=\bigcap_{\fq\in \Spec^1(R)}\fp R_{\fq}={\fp}R_{\fp}\cap\bigg(\bigcap_{\fq\in\Spec^1(R)\setminus\{\fp\}}\fp R_{\fq}\bigg)  = \fp R_{\fp}\cap R=\fp.$$
	In view of Corollary \ref{decproc1}, $\fp$ is free. Thus, $\fp$ is principal. Since any height-one prime ideal is principal, it follows that $R$ is $\UFD$. Suppose inductively that the claim holds for $d-1$. This implies that $\fp$ is locally free.
	Similar to the case $d=4$, we know $\depth_R (\Hom_{R}(\fp,\fp))\geq 4$ and  also $\fp$ is reflexive. 
	In view of Corollary \ref{decproc1}, $\fp$ is  free and so principal. In sum, $R$ is $\UFD$. \end{proof}

\begin{fact}\label{samu}(Samuel, see \cite{sam})
Let $R=\bigoplus_{n=0}^{\infty}R_n$ be  an integral domain and $Q$ be the fraction field of $R_0$. Then $R$ is $\UFD$ if and only if $R_0$ is $\UFD$, each $R_n$ is reflexive and $ R\otimes_{R_0} Q$ is $\UFD$.
\end{fact}

\begin{corollary}\label{regufd}
Let $(R,\fm)$ be a $2$-dimensional regular local ring
and $M$ be finitely generated. Then $\Sym_R(M)$	is regular if and only if $\Sym_R(M)$ is $\UFD$.
\end{corollary}

\begin{proof}  Assume that $\Sym_R(M)$ is $\UFD$.
	By Fact \ref{samu} $M$ is reflexive. Since
	$R$ is $2$-dimensional and regular, $M$ is free. Let $n:=\rank M$. Then
	$ \Sym_R(M)=R[X_1,\ldots,X_n]$ which is regular.
\end{proof}

The $2$-dimensional assumption is important:

\begin{example}\label{sym}
Let $(R,\fm,k)$ be a $3$-dimensional regular local ring and $M:=\Syz_2(k)$.
Then $\depth(\Sym_n(M))=2$ for all $n>0$. In particular, $\Sym_R(M)$ is a nonregular $\UFD$.
\end{example}

\begin{proof} Let $\{x,y,z\}$ be a regular parameter sequence of $R$.
In view of \cite[Proposition 5]{sam}	we know that $\Sym(M)\cong \frac{R[X,Y,Z]}{(xX+yY+zZ)}$ is $\UFD$. Due to Fact \ref{samu} we deduce that
$\Sym_n(M) $ is reflexive.  Since  the following relation $$(xX+yY+zZ) R[X,Y,Z]_{n-1}$$ is   nontrivial,  we deduce that $\Sym_n(M) $ is not free. Thus
 $\depth(\Sym_n(M))=2$. We set  $$\fn:=\fm\oplus (\oplus_{n>0}\Sym_n(M)).$$As $xX+yY+zZ\in\fn^2$, we see
 $\Sym_R(M)$ is  not regular.
\end{proof}

Samuel constructed a reflexive module $M$ such that
$\Sym_2(M)$ is not reflexive.
Here, there is another one:

\begin{example}\label{four}
	Let $(R,\fm,k)$ be a $4$-dimensional regular local ring.
	Then $M:=\Syz_2(k) $ is reflexive. In view of \cite[Proposition 3.1.11]{wol2},    $\pd(\Sym_2(M))=4$.
	By Auslander-Buchsbaum formula,
	$\depth(\Sym_2(M))=0$. In particular, $\Sym_2(M)$ is not reflexive.
\end{example}

\begin{conjecture}
(See \cite[Conjecture 6.1.4]{wol2}) Let $R$ be a regular local ring. If $\Sym_R(M)$ is $\UFD$, then $\pd(M)\leq 1$.
\end{conjecture}

Samuel \cite{sam} remarked that there is no symmetric   analogue of Auslander's theorem
on torsion part of tensor powers.  We present a tiny remark:

\begin{proposition}\label{obsym}
	Let $(R,\fm)$ be a regular local ring  and $M$ be    of rank one. If
	$\Sym_n(M)$ is reflexive for some  $n\geq  \max\{2,\dim R\}$, then $M$ is free.
\end{proposition}

\begin{proof}We   proceed by induction
	on $d:=\dim R$. Suppose $d=1$. Due to the
	structure theorem for finitely generated modules over $\DVR$,
	$M=F\oplus T$ where $F$ is free and $T:=\bigoplus_{i=1}^t \frac{R}{\fm^{n_i}}$ is the torsion part.
	Suppose on the contradiction that $T\neq0$, and recall that	$$\Sym_n(M)=\bigoplus_{i+j=n}\Sym_i(F)\otimes_R\Sym_j(T)\cong L_1\oplus\big( L_2\otimes\Sym_1(\frac{R}{\fm^{n_i}})\big)=L_1\oplus( L_2\otimes \frac{R}{\fm^{n_1}}),$$ where $L_2\neq 0.$
Hence, $\Sym_n(M)$ is not reflexive.
This contradiction says that   $T=0$  and consequently, $M$ is free. Suppose now that $d>1$
and the claim holds for all regular rings of smaller dimension.
Let $\fp\in(\Spec(R)\setminus\max(R))$, and recall that $\Sym_n(M)_{\fp}\cong \Sym_n(M_{\fp})$
	is reflexive over $R_{\fp}$. Inductively, $M_{\fp}$ is free  over $R_{\fp}$.
	In particular, we may assume that $M$ is locally free over the punctured spectrum.
	
There is a surjective map $M^{\otimes n}\stackrel{\pi}\lo \Sym_n(M)\to 0$.
Set $K:=\ker(\pi)$. 
Again, let $\fp$ be a prime ideal which is not maximal. Then, we have
$$0\lo K_{\fp}\lo (M^{\otimes n})_{\fp}\lo \Sym_n(M)_{\fp}\lo 0\quad(\ast)$$
Since $M$ is locally free of rank  $1$,  we have $$\Sym_n(M)_{\fp}\cong\Sym_n(M_{\fp})\cong (R_{\fp}[X])_n.$$  
This is  free over $R_{\fp}$ and is of rank $1$. Also, $(M^{\otimes n})_{\fp}$
is  free over $R_{\fp}$ and is of rank $1$. In particular,
the sequence $(\ast)$ splits, i.e., $K_p\oplus R_{\fp}\cong R_{\fp}$.
It follows that $K_p=0$. Hence, $K$ is supported in $\fm$. In particular, it is of finite length. According to Grothendieck's vanishing theorem, we know $\HH^+_{\fm}(K)=0$. This yields the following exact sequence
$$0=\HH^1_{\fm}(K)\lo \HH^1_{\fm}(M^{\otimes n})\lo \HH^1_{\fm}(\Sym_n(M))\lo \HH^2_{\fm}(K)=0 \quad(+)$$
Another use of Fact \ref{samu} implies that $\Sym_n(M)$ is reflexive.
Since $\dim R>1$, we deduce that  $\depth(\Sym_n(M))\geq2$. Due to the cohomological characterization of depth,
 $\HH^1_{\fm}(\Sym_n(M))=0$. In view of (+) we see $\HH^1_{\fm}(M^{\otimes n})=0$. 
Now, we recall the following recent result:
\begin{enumerate}
	\item[Fact A)]  (See \cite{acs}) 	Let  $(A,\fn)$  be a local complete intersection ring and let $N$ be   rigid  and locally free. Let $m\geq  \max\{2,\dim A\}$. If $\HH^i_{\fn}(N^{\otimes m})=0$ for some $0\leq i\leq\depth_R(N)$, then $N$ is free.
\end{enumerate} 
Over regular rings any module is rigid.
Therefore, Fact A) implies that $M$ is free.
 \end{proof}

In the case of ideals the following fact is due to  Micali and Samuel \cite{sam}:
\begin{fact}\label{mi}
	Let $M$ be a module of rank one. If
	$\Sym_R(M)$ is $\UFD$, then $M$ is projective.
\end{fact}

\begin{proof} We may assume that $R$ is local.
	By Fact \ref{samu} $M$ is reflexive and $R$ is $\UFD$. Since $R$ is normal, $\Cl(R)=\pic(R)$. Recall that
	$\pic(R)$ is the isomorphism classes of rank $1$
	reflexive modules. Since $\Cl(R)=0$, $M$ is free.
 \end{proof}

\begin{corollary} \label{ob3}Let $(R,\fm)$  be a $3$-dimensional regular local ring,
	 $M$ be torsion-free and of rank one. If
	$\Sym_n(M)$ is reflexive for some  $n>0$, then $M$ is free.
\end{corollary}

\begin{proof}
The case $n=1$ (resp. $n\geq 3$ ) is in   Fact \ref{mi} (resp. Proposition \ref{obsym}).
Without loss of the generality we assume that $n=2$.
By the proof of  Proposition \ref{obsym}, we assume that $M$ is locally free and 
that  $\HH^1_{\fm}(M^{\otimes 2})=0$. Since $M$ is torsion-free, it follows that
$M$ is free.
\end{proof}

The rank $1$ condition is important, see  Example \ref{sym}.

\section{Reflexivity of ideals} 

In this section we study some aspects of the  following problem and also its assumption, e.g., the artinian and et cetera:
\begin{problem}
	Let  $(R,\fm)$ be artinian and suppose $\fm^n\neq 0$ is  reflexive. Find $n$ such that  $R$ becomes Gorenstein.
\end{problem}

\begin{proposition}\label{29}  Let $(R,\fm,k)$ be a local artinian ring. Assume that there exists a non-negative integer $n$
	such that $\fm^n\neq 0$ and $\fm^{n+1}=0$. If $\fm^n$ is reflexive, then $R$ is Gorenstein.
\end{proposition}
The proof shows $ (\fm^n)^*  $ is cyclic.
\begin{proof} Consider the minimal presentation of $\Hom_{R} (\fm^n ,R)$
	$$ R^{n_1} \longrightarrow R^{n_0} \longrightarrow (\fm^n)^* \longrightarrow 0 \quad(+)$$This
gives	the exact sequence  $ 0\to (\fm^n) \to R^{n_0} \overset{f}\longrightarrow R^{n_1},$   where $f:=(a_{ij})_{{n_{1}}
		\times {n_{0}}}$ with $ a_{ij} \in \fm$.
 and so $f(\Soc (R^{n_0}))=0$.
	Now, $$\Soc(R)^{\oplus n_{0}}=\Soc (R^{n_{0}}) \subseteq \ker(f)\cong \fm^n .$$ As $\fm^{n+1}=0 $, one has  $\fm^n \subseteq \Soc(R)$. Consequently,
	 $ \Soc (R)^{\oplus n_{0}} \subseteq \fm^n \subseteq \Soc(R).$ 
This implies that $ n_0 =1$ and	$ \fm^n =\Soc (R) $. We plug  the first one into $(+)$ and observe that $ (\fm^n)^*  $ is cyclic.
 Applying $(-)^\ast$ to it, yields that
	$\Soc (R)^*$ is cyclic. Recall from its definition,
	$ \Soc (R) $ is an  $R/\fm$-vector space, and by denoting its dimension with $ t := \dim_{k} (\Soc(R)) $ we have:  $$ \Soc(R)^* = \Hom( \oplus _t k,R) = \oplus_t \Hom(R/\fm , R) = \oplus_t \lbrace r
	\in R \vert r\fm =0 \rbrace = \oplus_t \Soc (R) .$$
	Since $\Soc (R)^*$ is cyclic we know $ t=1 $  and also  $ \Soc(R) $ is cyclic. In other words,
	type of $R$ is one. It remains to note that Cohen-Macaulay rings of type one are Gorenstein.
\end{proof}

\textbf{Second proof of Proposition \ref{29}.} Since $\fm$ annihilates $\fm^n$, we observe that $\fm^n$ is a $k$-vector space.
Since any directed summand of a reflexive module is again reflexive,
we deduce that $R/\fm$ is reflexive. By the first argument, we have $$ R/\fm\cong \bigoplus_{\Soc(R)}\bigoplus_{\Soc(R)} R/ \fm,$$ i.e., $\type(R)=1$, and so $R$ is Gorenstein. $\Box$

\begin{example}
	Let $R:=\frac{k[x,y]}{(x^3,x^2y,xy^2,y^4)}$. Then
$\fm^3$ is not reflexive.
\end{example}

\begin{proof}Since $\{xy,y^3\}\subseteq\Soc(R)$, $R$  is not Gorenstein.
 Recall that	$\fm^4=0$ and $\fm^3=y^3R\neq 0$.  
 Proposition \ref{29} shows that $\fm^3$ is not reflexive. 
\end{proof}

\begin{fact}\label{30}  Let $(R,\fm)$ be a local ring.  If $\depth R\geq 2$, then $\fm^n$ is not reflexive for
	all $n\geq 0$.
\end{fact}

\begin{proof} On the contrary, suppose that $\fm^n$ is reflexive for some natural number $n$. As $\fm^n$ is a reflexive $R$-module
	and $\depth R\geq 2$, by \cite[Proposition 1.4.1]{BH}, we conclude that $ \depth_R \fm^n \geq 2 $.
	The short exact sequence $0 \to \fm^n \to R \to R/\fm^n \to 0$
	yields the following long exact sequence of local cohomology modules:
	$$  0 \rightarrow \text{H}_{\fm}^0 (\fm^n) \rightarrow \text{H}_{\fm}^0 (R) \rightarrow \text{H}_{\fm}^0 (R/\fm^n)
	\rightarrow \text{H}_{\fm}^1 (\fm^n) \rightarrow \text{H}_{\fm}^1 (R) \rightarrow \text{H}_{\fm}^1 (R/\fm^n)\rightarrow \cdots .$$
	As $\depth R\geq 2$, we have $\text{H}_{\fm}^0(R)=\text{H}_{\fm}^1(R) =0$, and so $\text{H}_{\fm}^1 (\fm^n)\cong \text{H}_{\fm}^0 (R/\fm^n)=
	R/\fm^n \neq 0 .$
	This implies that $\depth_R(\fm^n)=\inf \lbrace i\in \mathbb{N}_0 \mid \text{H}_{\fm}^i(\fm^n) \neq 0 \rbrace \leq 1.$
	This is a contradiction.
\end{proof}

Recall that Bass \cite[Theorem 6.2]{bass} proved that an artinian local ring is Gorenstein iff
all
of its ideals are reflexive.
\begin{proposition}\label{fa}  Let $(R,\fm)$ be a local ring. The following assertions are true:
	\begin{itemize}
		\item[(i)]  If $\dim R=0$ and $\fm$ is reflexive, then $R$ is Gorenstein.
		\item[(ii)]  If $\depth R\geq 2$, then $\fm\ncong \fm^{**}$. In fact, $\fm^n$ is not reflexive for all $n>0$.
		\item[(iii)]  If $R$ is quasi-reduced and $\depth R=1$, then $\fm\cong \fm^{**}$.
	\end{itemize}
\end{proposition}

\begin{proof}
i)	Apply $(-)^{*}$ to the short exact sequence
 $0\to \fm  \overset{i}\longrightarrow R \to R /\fm \to 0$  implies the exact sequence  fits in the following diagram:

	$$
\begin{CD}
 	0@>>> \Hom(R / \fm  , R) @>>> R^\ast @> >> \fm^\ast@>>>\Ext_{R}^1(R/\fm,R)  @>>> 0\\
@.=@AAA\cong @AAA = @AAA = @AAA  \\
 	0@>>> \bigoplus_{i=1}^{ t}  R/\fm @>>> R @>>>\fm^\ast@>>> \bigoplus_{i=1}^{ t_1}  R/\fm @>>>0,\\
\end{CD}
$$where $t$ is the type of $R$ and $t_1$ is an integer.
Let us break down the bottom sequence   into the following  short exact sequences:	
\begin{itemize}
	\item[(1)]  $0\to \bigoplus_{i=1}^{ t}  R/\fm \to R\to L\to 0$,
	\item[(2)]   $0\to L    \to \fm^\ast\to \bigoplus_{i=1}^{ t_1}  R/\fm\to 0$.
\end{itemize}
Taking duality from 2) yields $0\lo\bigoplus_{i=1}^{ t} \bigoplus_{i=1}^{ t_1}  R/\fm   \lo \fm\cong\fm^{\ast\ast}.$ 
In particular, $0\to\bigoplus_{i=1}^{ t} \bigoplus_{i=1}^{ t_1}  R/\fm \subseteq R$. Taking socle and computing
the length yields that  $ tt_1\leq t$. Since $t>0$ we deduce that $t_1\leq 1$. We have two possibilities:
(a) $ t_1=0$, and (b)   $ t_1\neq0$.
\begin{itemize}
	\item[(a)] Assume $ t_1=0$. Recall that $\mu_i:=\dim\Ext_{R}^i(R/\fm,R)$ is the $i$-th Bass' number and $t_1=\mu_1$.
	Suppose on the way of contradiction that $\id(R)=\infty$. In this case Bass \cite[Lemma 3.5]{bass} proved that $\mu_i>0$ for all $i>\dim R$.  Since $t_1=0$ we get a contradiction. So,  $\id(R)<\infty$, i.e., $R$  is Gorenstein.
	\item[(b)]  Assume  $ t_1\neq0$. This yields that $t_1=1$. Then, we have   $0\to R/ \Soc(R)    \to \fm^\ast\to   R/\fm\to 0$. Recall that  $\Hom( R/ \Soc(R) ,R)=\{r\in R:r\Soc(R)=0\}=\fm.$ First, we take  another star and then  apply
	the last observation to deduce the following exact sequence  $0 \lo \bigoplus_{i=1}^{ t}  R/\fm \lo \fm^{\ast\ast}\lo \fm\stackrel{\phi}\lo \Ext^1_R(R/\fm ,R) = \bigoplus_{i=1}^{ t_1}  R/\fm =R/ \fm.$ 
In other words: $$0 \lo \bigoplus_{i=1}^{ t}  R/\fm \lo \fm^{\ast\ast}\lo \fm\lo\im{\phi}\lo 0.$$Recall that
$\im \phi\subset R/\fm$. This says  that the  length of $\im \phi$ is at most one. Let us take the length and use its additivity.
This yields
 \[\begin{array}{ll}
0&=\ell(\bigoplus_{i=1}^{ t}  R/\fm )-\ell(\fm^{\ast\ast})+\ell(\fm)-\ell(\im\phi)\\&=t-\ell(\im\phi),\\
\end{array}\]here, we used the fact that $\fm$ is reflexive.
So,  $0<t=\ell(\im\phi)\leq 1,$ 
i.e., type of $R$ is one and so $R$ is Gorenstein.
\end{itemize}

The proof is now complete. 
	
ii)	By Fact \ref{30}, $\depth(\fm)=1$. Since $\depth(\fm^{**})\geq 2$ we have $\fm\ncong \fm^{**}$.

iii) Having Fact \ref{mas} in mind, the argument is a slight modification of \cite[Proposition 4.1]{F}. We leave the routine details to the reader.
\end{proof}

\begin{corollary}
Let $R$ be $(\Se_2)$, and $\fp$ be a prime ideal of height at least two. The following holds:	\begin{itemize}
	\item[(i)] $\fp^n$ is not reflexive for all $n>0$.
	\item[(ii)]  $\fp^{(n)}$ is not reflexive for all $n>0$.
\end{itemize}
\end{corollary}
\begin{proof}
	Suppose on the way of contradiction that $\fp^n$ (resp. $\fp^{(n)}$) is reflexive.
	Then $\fp^nR_{\fp}$ (resp. $\fp^{(n)}R_{\fp}$) is reflexive.
	Recall that $\fp^{(n)}R_{\fp}=\fp^nR_{\fp}$,
	and from Serre's condition that $\depth(R_{\fp})\geq 2$. It remains to apply the previous result and get to a contradiction.
\end{proof}
 
Here, we present a sample with $\depth R=0$:
\begin{example}\label{neq2}
Let $S$ be any regular
local ring of dimension $n>0$ and  $\fp$ be a  prime ideal of $S$. Let  
$R:=S/\fp \fm$. 
	Then $\fm\cong \fm^{\ast\ast}$ if and only if $n=1$.
\end{example}

\begin{proof}
Without loss of generality, we may assume $\fp\neq 0$. First, assume that $\Ht(\fp)\geq2$.
According to a result of Ramras,
$R$ is $\BNSI$. In particular, any reflexive module is free. This yields that $\fm\cong \fm^{\ast\ast}$ if and only if $n=1$. Now, we consider to the case $\Ht(\fp)=1$. Since  $S$ is $\UFD$,
$\fp=(x)$ for some $x$. Without loss of generality we assume that $n>1$, and we are going to show $\fm$
is not reflexive. Suppose on the way of contradiction,
$\fm$ is reflexive and search a contradiction. 	
Recall that $$\Hom_R(k,R)=\{r:r\fm=0\}=xR\quad(\ast)$$Apply $(-)^{*}$ to short exact sequence
	 $0\longrightarrow \fm  \overset{i}\longrightarrow R \longrightarrow R /\fm \longrightarrow 0,$    gives us:

$$
\begin{CD}
0@>>> \Hom(R / \fm  , R) @> >> R^\ast @>>> \fm^\ast@>>>\Ext_{R}^1(R/\fm,R)  @>>> 0\\
@.=@AAA\cong @AAA = @AAA = @AAA  \\
0@>>> xR @>>> R @>>>\fm^\ast@>>> \bigoplus_{i=1}^{ t_1}  R/\fm @>>>0,\\
\end{CD}
$$where $t_1$ is an integer.
In other words,  the sequence 	 $0\lo R/xR    \lo \fm^\ast\lo \bigoplus_{i=1}^{ t_1} k\lo 0$ is exact.
We take another dual and use   $\Hom_R(R/xR,R)=\{r:rx=0\}=\fm.$ Let us  put things into the following diagrams:
	
	$$
	\begin{CD}
	0@>>> \bigoplus_{i=1}^{ t_1}\Hom(R / \fm  , R) @>>> \fm^{\ast\ast} @>>> (R/xR )^\ast@>>>\bigoplus_{i=1}^{ t _1}\Ext_{R}^1(R/\fm,R)   \\
	@.=@AAA\cong @AAA = @AAA = @AAA  \\
	0@>>> \bigoplus_{i=1}^{ t_1 }xR @>>> \fm @>>>\fm@>>>\bigoplus_{i=1}^{ t _1} \bigoplus_{i=1}^{ t _1}  R/\fm . \\
	\end{CD}$$
On the one hand, there is an equality
$(0:x)=\fm$. From this,  $\Soc(xR)=\Soc(R/(0:x))=\Soc(k)=k,$  and consequently 
 $\dim(\Soc(\bigoplus_{i=1}^{ t_1}xR)) = t_1\dim(\Soc(xR))=t_1. $	
	On the other hand, $\bigoplus_{i=1}^{ t_1 }xR \subseteq R$. By taking length from
$\bigoplus_{i=1}^{ t _1}xR\subseteq \Soc(R)=xR$  we deduce $t_1\leq1$. In view of \cite[Page 373]{BH}
we observe  $1\geq t_1=\mu_1(\fm,R)\geq \dim R=n-1\geq 1.$ This implies  $n=2$.	
In this case $\dim R=1$. In sum, we proved that  $\mu_{\dim R} (\fm, R)=1.$  
This allows us to apply  a beautiful result of Roberts (see \cite[9.6.3]{BH}) to deduce that $R$ is Cohen-Macaulay, i.e.,
 $0=\depth(R)=\dim R=1.$  This is a contradiction that we searched for it.
\end{proof}

\begin{remark}(Souvik Dey)
	i) In  Proposition \ref{fa}(i), one may weaken $\dim R=0$ to $\depth R=0$. Indeed, this follows from \cite[Theorem 4.1(3)]{dey} by a different argument.
	
	ii) In Proposition \ref{fa}(iii), no quasi-reduced hypothesis is needed, see \cite[Theorem 4.1(2)]{dey}. 
	
		iii) Example \ref{neq2} {fa}(iii) extends to more general setting \cite[Theorem 4.1(3)]{dey}. 
\end{remark}

\begin{corollary}
	Let $R:=k[[X,Y]]	/{X (X,Y)}$. Then   $\fm^n$ is not reflexive for all $n>0$.
\end{corollary}

\begin{proof} In view of   Example \ref{neq2} we know $\fm $ is not reflexive. Now, let $n>1$ and suppose on the way of contradiction that  $\fm^n=y^nR$ is reflexive. In particular, $$(\fm^n)^{\ast}\cong(\fm^n)^{\ast\ast\ast}\quad(\natural)$$ Recall that
	$$(\fm^n)^\ast=\Hom(R/(0:y^n),R)=\Hom(R/xR,R)\cong\{r\in R:rx=0\}=\fm\quad(\dagger)$$Taking two times dual from $(\dagger)$  yields $$(\fm^n)^{\ast\ast\ast}\cong\fm^{\ast\ast}  \quad(\sharp)$$Combining theses together
	$$\fm\stackrel{\dagger}\cong(\fm^n)^{\ast}\stackrel{\natural}\cong(\fm^n)^{\ast\ast\ast}\stackrel{\sharp}\cong\fm^{\ast\ast}\quad(+)$$Then $(+)$ shows that $\phi_{\fm}:\fm\to \fm^{\ast\ast}$ is an isomorphism. By definition, $\fm$ is reflexive. This is in contradiction with   Example \ref{neq2}. So, $\fm^n$ is not reflexive.	
\end{proof}

\begin{discussion}\label{n}Reflexivity of  $\fm$ does not imply  reflexivity of ideals:
\begin{itemize}	\item[ i)] Let $(R,\fm)$ be a $1$-dimensional reduced ring which is not Gorenstein.
	Then $\fm$ is reflexive and there is an ideal which is not reflexive. Indeed, thanks to Proposition \ref{fa}(iii) we know $\fm$ is reflexive. Recall that $1$-dimensional reduced rings
	are Cohen-Macaulay.
	It remains to recall from \cite[(6.2) Theorem]{bass} that any one-dimensional  Cohen-Macaulay ring is Gorenstein
	if and only if   every ideal   is reflexive.
	
\item[ ii)] Concerning to the first item, canonical module is the proposed ideal.
	To see an explicit example, let $R:=k[[x^3,x^4,x^5]]$. Then $\omega_R=(x^3,x^4)$ is not reflexive.\end{itemize}
\end{discussion}

 The Gorenstein property  of a $1$-dimensional  ring
may follow by observing that $\fm$ is totally reflexive. This involves to checking
infinitely many vanishing of Ext-modules.
Vanishing of some initial $\Ext$-modules  may yield the Gorenstein property:

\begin{remark}\label{31}  Let $(R,\fm)$ be a 1-dimensional Cohen-Macaulay local ring.  Then $R$ is Gorenstein
	if and only if $\Ext_R^1(\fm,R)=\Ext_R^1((\fm^2)^\ast,R)=0$.
\end{remark}

\begin{proof} First, assume that $R$ is Gorenstein. For any nonzero submodule $L$ of a finite rank free $R$-module, we have
	$\depth_RL\geq 1$, and so the Auslander-Bridger formula implies that $$\sup\{i\in \mathbb{N}_0\mid \Ext_R^i(L,R)\neq 0\}=
	\Gdim_RL=0.$$
	Conversely, we assume   $\Ext_R^1(\fm,R)=\Ext_R^1((\fm^2)^*,R)=0$, and we are going to show  $R$ is Gorenstein. By \cite[Proposition 4.1]{F}, we know that $\fm^{**}\cong
	\fm $. Set $\mu:=\text{vdim}_{k}(\fm /\fm^2)$ and $t:=\text{vdim}_{k}(\Ext_{R}^1(R/\fm ,R))$. The short exact sequence
 $0\to \fm^2 \overset{i}\to \fm \to \fm /\fm^2 \to 0$  implies the exact sequence
	$$0 \longrightarrow \Hom(\fm / \fm^2 , R) \longrightarrow \fm^* \longrightarrow (\fm^2)^* \longrightarrow
	\Ext_{R}^1(\fm/\fm^2,R) \longrightarrow \Ext_{R}^1 (\fm, R).$$  By  definition of  $\mu$ and $t$, one has 
 $\Ext_{R}^1 (\fm/\fm^2,R)\cong \bigoplus_{i=1}^{t \mu}  R/\fm.$  Since $\depth_R(R)>0$, we can easily see that  
	$\Hom(\fm /\fm^2,R)=0$. We combine this along with $\Ext_R^1(\fm,R)=0$ to deduce the following short exact sequence
 $0 \to \fm^* \overset{i^*}\to (\fm^2)^*\lo \bigoplus_{i=1}^{t \mu}  R/\fm \to 0.$ 
	Since  $(R/\fm)^*=0=\Ext_{R}^1((\fm^2)^*,R),$  and via applying the functor $(-)^{*}$ to the later exact sequence, it yields 
	the following exact sequence
	$$0 \longrightarrow (\fm^2)^{**} \overset{i^{**}}\longrightarrow \fm^{**} \longrightarrow \bigoplus_{i=1}^{t^2 \mu}  R/\fm  
	\longrightarrow 0.$$ Now recall that  $\fm^{**}\cong
	\fm $. 
	In other words, $\fm/(\fm^2 )^{**}\cong \bigoplus_{i=1}^{t^2 \mu}  R/\fm $.
	By definition, $\fm/\fm^2=	\bigoplus_{i=1}^{  \mu}  R/\fm$. Let us recall    from $ \fm^2\stackrel{\subseteq}\lo(\fm^2 )^{**}\subseteq \fm$ that 
	$\fm/\fm^2 \to    \fm/(\fm^2 )^{**} \to 0$.
	We put all things together and we get to
	the following diagram:$$\xymatrix{
		&& \fm/\fm^2 \ar[d]_{=}\ar[r]^{}&  \fm/(\fm^2 )^{**} \ar[d]^{=} \ar[r] &0&\\
		&&	\bigoplus_{i=1}^{  \mu}  R/\fm\ar[r]_{f }& 	\bigoplus_{i=1}^{t^2 \mu}  R/\fm.\\
		&&&}$$
	We deduce from this diagram that $f$ 
	is surjective. This yields the  validity of  $ \mu \geqslant \mu t^2 $. Consequently,  $ 0<t\leq 1 $. Therefore, $ R $ is Gorenstein.
\end{proof}

One may apply a famous  result of Roberts and simplifies  the above Remark:
\begin{remark}(Souvik Dey)
We can drop the Cohen—Macaulay, and the vanishing of $\Ext^1_R((\fm^2)^*, R)$, and just assume $\depth R=1$ and $\Ext^1_R(\fm, R)=0$. Indeed, this implies $\Ext^2_R(k, R)=0$, hence $R$ is Gorenstein by using another result of Roberts, see \cite[Theorem II]{Ro2}.\end{remark}

\section{Questions by  Holanda and Miranda-Neto}

The following question was asked in \cite[Question 5.24]{nano}: \begin{question}\label{11.1}
Let R be a Cohen-Macaulay local ring with canonical module $\omega_R$. If $\pd_R(\omega_R^\ast)
<\infty$, must R be Gorenstein?
\end{question}

First, we deal with low-dimensional cases:
\begin{proposition}\label{11.2}
	Suppose $\dim(R)\leq 3$.	Then   Question \ref{11.1} is satisfied.
\end{proposition}

\begin{proof}
  Suppose first that $d:=\dim R=0$. Recall that $\pd_R(\omega_R^\ast)
	<\infty$. This allows us to apply   Auslander-Buchsbaum formula, to observe that
	$$\pd_R(\omega_R^\ast)=\depth(R)-\depth(\omega_R^\ast)=0.$$By the mentioned result of Kaplansky, $\omega_R^\ast $ is free. Since $\depth(R)=0$, and in view of Ramras' result from subsection 4.4,
	$\omega_R  $ is free. Thus, $R$ is Gorenstein.
	
	Suppose $d=1$.  By a formula of Auslander-Goldman,
	$$\depth(\Hom(\omega_ R,R))\geq \min\{2,\depth(\omega_ R)\}=1.$$ By another use of Auslander-Buchsbaum formula, $\omega_ R^\ast$ is free.
	Now, recall
	
	\begin{itemize}	\item[ Fact]i):  
		(See \cite[Lemma 3.9]{dao}) Suppose $R$ is Cohen-Macaulay with dimension   1, and suppose $M\in CM(R)$.
		Then $M^\ast$
		free implies $M$ is free.
	\end{itemize}
	So, $\omega_ R$ is free. Consequently, $R$ is Gorenstein.

	Suppose $d=2$. The assumptions behave well with respect to the localization. According to the previous cases, we may and do assume that $R$ is $(G_1)$. So, $R$ is quasi-normal. Over any quasi-normal ring, and by the mentioned result of Vasconcelos, a module $(-)$ is reflexive iff any $R$-regular sequence of length at most two is $(-)$-regular. So, $\omega_ R$ is reflexive. By a formula of Auslander-Goldman,
	$$\depth(\Hom(\omega_ R,R))\geq \min\{2,\depth(\omega_ R)\}=2.$$ By another use of Auslander-Buchsbaum formula, $\omega_ R^\ast$ is free.
	 So, $\omega_ R\cong \omega_ R^{\ast\ast}$ is free. Consequently, $R$ is Gorenstein.
	 
	 Suppose $d=3$. It is easy to see that  $\omega_ R^\ast$ is locally free over punctured spectrum, of projective dimension at most one, and reflexive. So, there is an exact sequence $$0\lo R^n\lo R^m\lo \omega_ R^\ast\lo 0.$$ Applying $-\otimes_R\omega_ R$ to it gives us $$0=\Tor_1^R(\omega_R,R^m)\lo\Tor_1^R(\omega_R,\omega_R^\ast)\lo R^n\otimes_R\omega_R\lo R^m\otimes_R\omega_R\lo\omega_R\otimes \omega_ R^\ast\lo 0.$$
	 Since $\Tor_1^R(\omega_R,\omega_R^\ast)$
	 is of finite length and $ \omega_R^n$
	is of positive depth, then $\Tor_1^R(\omega_R,\omega_R^\ast)=0$.
	 The long exact sequence of local cohomology module shows that $$0=H^1_\fm(\omega_R^m)\lo H^1_\fm(\omega_R\otimes \omega_ R^\ast)\lo H^2_\fm(\omega_R^n)=0,$$becuase $\depth(\omega_R^n)=\depth(\omega_R^m)=\dim(R)>2$.
	 
	 	Now, we recall the following:
	 
	 \begin{itemize}	\item[ Fact]ii): 
	 	(See \cite[Corollary 8.2]{acs}) Suppose $\depth(R)>0$  and  $M   $ is locally free on punctured spectrum.
	  If $H^1_\fm(M\otimes  M^\ast) =0$, then $M$ is free.
	 \end{itemize}
	 So, $\omega_ R$ is free. Consequently, $R$ is Gorenstein.
\end{proof}
One may derive the next result from Foxby's correspondence. In particular, the next proof suggests further applications of local cohomology to 
Auslander-Foxby's correspondence.
\begin{theorem}
 The desired property of Question \ref{11.1} is true.
\end{theorem}

\begin{proof}
We proceed by induction on $d:=\dim R$ which is finite by the local assumption. The case $d=0$ is in Proposition \ref{11.2}.   
The assumptions behave well with respect to the localization. In particular, we are able to apply induction hypothesis to deduce that $R$ is Gorenstein on the punctured spectrum.

	 \begin{itemize}	\item[ Fact]: 
	(See \cite[Proposition 3.4]{finitsup})
Let $(R,\fm,k)$ be a  local ring with an ideal $\fa$, $M$ and $N$  be  such that $\pd(M)<\infty$ and  one of them is locally free over
$\Spec(R)\setminus\V(\fa)$. Let $0\leq r<d:=\dim R$ be such that
$\grade_R(\fa,M)+\grade_R(\fa,N)\geq d+r+1$. Then   $\HH_{\fa}^0(M\otimes_RN)=\ldots=\HH_{\fa}^r(M\otimes_RN)=0$.
 \end{itemize}
According to the previous result, we may assume that $d>3$. In particular, $\depth(\omega_R^\ast)\geq 2$.
In the previous fact, apply $\fa:=\fm$, $M:=\omega_R^\ast$, $N:=\omega_R$, and $r:=1$. Recall that  $\grade_R(\fa,M)+\grade_R(\fa,N)\geq d+1+1$. So, the above fact  can be applied. It gives
$\HH_{\fm}^1(\omega_R\otimes \omega_ R^\ast)=0$.
In the light of Fact ii) from Proposition \ref{11.2}, $\omega_ R$ is free. Consequently, $R$ is Gorenstein.
\end{proof}

The following question was asked in \cite[Question 3.17]{nano}: \begin{question}\label{11.4}
	Let $R$ be a Cohen-Macaulay local ring with canonical module $\omega_R$ and dimension
	$d$. Let $M$ be a finite $R$-module satisfying $(S_k)$ with either $k = 1$ or $3 \leq k < d$, and such that
	$\Ext^i_R(M,\omega_R) = 0$ for all $i = 1, \dots ,\max\{1, d-k\}$. It is true that $M$ is maximal Cohen-Macaulay?
\end{question}
\begin{observation}
		The answer to Question \ref{11.4} is yes.
\end{observation}

\begin{proof}We may assume the ring is complete. There is nothing to prove if $\depth(M)=d$.
So, let $\depth(M)\leq i<d$. By local duality, $H^i_{\fm}(M)=\Ext^{d-i}
_R(M,\omega_R)^v $
which is zero by the assumption. By homological characterization of $\depth$, we know $H^i_{\fm}(M)=0$ for all $i<\depth(M)$. Combining these, $H^i_{\fm}(M)=0$ for all $i<d$. So, $M$ is maximal Cohen-Macaulay.
\end{proof}

The following question was asked in \cite[Question 4.7]{nano}: \begin{question}\label{11.6}
	Let $R$ be a Cohen-Macaulay local ring with canonical module $\omega_R$ and dimension
	$d>0$. If
	$\Ext^+_R(\omega_R^\ast,R) = 0$, must $R$ be Gorenstein?
\end{question}

We support it with the following two results:
\begin{observation}
Suppose $R$ is quasi-normal.	The desired property of Question \ref{11.6} is valid.
\end{observation}

\begin{proof}It is easy to see $\omega$ is reflexive. Let $F_\bullet:\cdots  \to  F_0\to \omega_R^\ast\to  0  $ be a free resolution.
Since	$\Ext^+_R(\omega_R^\ast,R) = 0$, the sequence $0\to \omega_R^{\ast\ast}\to F_0^\ast\to \ldots $ is exact. Since $\omega$ is reflexive, and by pinching $F_\bullet$ to $F_\bullet^\ast$, we see  $\omega$ is totally acyclic. This property behaves well with respect to reduction via regular sequences. Let $\underline{x}$ be a regular sequence of length $d$. Let $\overline{R}:=R/\underline{x}R$. Then there is $\omega_{\overline{R}}\cong\omega_R/\underline{x}\omega_R\subseteq \overline{F}$ where $\overline{F}$ is free. So, $\omega_{\overline{R}}$ is torsionless. By Corollary \ref{INF},
$\overline{R}$ is Gorenstein. Consequently, $R$ is as well.
\end{proof}

\begin{observation}\label{11.8}
	Suppose $R$ is of minimal multiplicity with infinite residue field.	The answer to Question \ref{11.6} is yes.
\end{observation}

\begin{proof} According  to the standard reduction
	by a regular  sequence, we may and do assume that $\dim R=1$. Suppose on the way of contradiction that $R$ is not Gorenstein. This allows us to apply Observation \ref{m3COR}
	to deduce that $\omega_R^\ast$ is free. 
	By Fact ii) from Proposition \ref{11.2} $\omega_R$ is free. Consequently, $R$
is Gorenstein.
\end{proof}

\begin{acknowledgement}
I thank   Alessio Caminata, Olgur Celikbas,   Souvik Dey, Jian Liu, Elham Mahdavy, Arash
Sadeghi,  and Saeed Nasseh for their comments on the earlier draft. I thank Rafael Holanda and 
Cleto Miranda-Neto for an email correspondence.
 I thank Mohammad Golshani for introducing  some elements of set theory.
\end{acknowledgement}

\end{document}